\documentclass[reqno,11pt,a4paper]{amsart}

\usepackage{amsmath}
  \usepackage{paralist}
  \usepackage{graphics} 
  \usepackage{epsfig} 
\usepackage{graphicx}  \usepackage{epstopdf}
 \usepackage[colorlinks=true]{hyperref}
\hypersetup{urlcolor=blue, citecolor=red}
\usepackage{amssymb, array, amsmath, amscd, mathtools}
\usepackage{graphicx,epstopdf} 
\usepackage[caption=false]{subfig} 
\usepackage[utf8]{inputenc}
\usepackage{amssymb}
\usepackage{bbm}
\usepackage{overpic}
\usepackage{xcolor}
\usepackage{ifthen}
\usepackage{cite}
\usepackage{tikz}
\usetikzlibrary{arrows.meta}
\usetikzlibrary{calc}
\tikzset{%
  ,
            base/.style = {rectangle, rounded corners, draw=black,
                           minimum width=4cm, minimum height=1cm,
                           text centered},
  activityStarts/.style = {base, fill=blue!30},
       startstop/.style = {base, fill=red!30},
    activityRuns/.style = {base, fill=green!30},
         process/.style = {base, minimum width=2.5cm, fill=white,},
}

\newcommand*{\keyterm}[1]{\emph{#1}}
\DeclareMathOperator{\re}{Re}			
\DeclareMathOperator{\diag}{diag}		
\DeclareMathOperator{\Span}{span}		
\renewcommand{\ker}{\mathcal{N}}					

\newcommand*{\ldelim}[2]{\csname#1l\endcsname#2}   
\newcommand*{\rdelim}[2]{\csname#1r\endcsname#2}   
\newcommand*{\mdelim}[2]{\csname#1m\endcsname#2}   

\newcommand{\ga}{\alpha}

\newcommand{\gl}{\lambda}
\newcommand{\gw}{\omega}

\newcommand*{\C}{{\mathbb{C}}}     

\newcommand*{\R}{{\mathbb{R}}}     
\newcommand*{\N}{{\mathbb{N}}}     

\newcommand*{\Lin}{{\mathcal{L}}}   

\usepackage{xcolor}

\newcommand{\bI}{{\mathbf I}}

\newcommand{\cA}{{\mathcal A}}
\newcommand{\cB}{{\mathcal B}}

\newcommand{\cG}{{\mathcal G}}

\newcommand{\cL}{{\mathcal L}}

\newcommand{\cN}{{\mathcal N}}

\newcommand{\cR}{{\mathcal R}}

\newcommand{\Arc}{A_{\mathrm{rc}}}
\newcommand{\Ad}{A_{\mathrm{d}}}

\newcommand*{\Abs}[2][default]{\ifthenelse{\equal{#1}{default}}{\left\lvert#2\right\rvert}{\ldelim{#1}{\lvert}#2\rdelim{#1}{\rvert}}}
\newcommand*{\Norm}[2][default]{\ifthenelse{\equal{#1}{default}}{\left\lVert#2\right\rVert}{\ldelim{#1}{\lVert}#2\rdelim{#1}{\rVert}}}

\newcommand*{\Iprod}[3][default]{\ifthenelse{\equal{#1}{default}}{\left\langle#2,#3\right\rangle}{\ldelim{#1}{\langle}#2,#3\rdelim{#1}{\rangle}}}
\newcommand*{\Dualpair}[3][default]{\ifthenelse{\equal{#1}{default}}{\left\langle#2,#3\right\rangle}{\ldelim{#1}{\langle}#2,#3\rdelim{#1}{\rangle}}}

\newcommand*{\Set}[2][default]{\ifthenelse{\equal{#1}{default}}{\left\{#2\right\}}{\ldelim{#1}{\{}#2\rdelim{#1}{\}}}}

\newcommand*{\dda}[3][1]{\ifthenelse{\equal{#1}{1}}{\frac{d#3}{d#2}}{\frac{d^{#1}#3}{d#2^{#1}}}}
\newcommand*{\ddb}[2][1]{\ifthenelse{\equal{#1}{1}}{\frac{d}{d#2}}{\frac{d^{#1}}{d#2^{#1}}}}
\newcommand*{\pd}[3][1]{\ifthenelse{\equal{#1}{1}}{\frac{\partial{#2}}{\partial{#3}}}{\frac{\partial^{#1}{#2}}{\partial#3^{#1}}}}
\newcommand*{\pdb}[2][1]{\ifthenelse{\equal{#1}{1}}{\frac{\partial}{\partial{#2}}}{\frac{\partial^{#1}}{\partial#2^{#1}}}}

\newcommand*{\inv}{^{-1}}

\newcommand{\pmat}[1]{\begin{pmatrix}#1\end{pmatrix}}
\newcommand{\pmatsmall}[1]{\begin{psmallmatrix}#1\end{psmallmatrix}}

\newcommand{\eq}[1]{\begin{align*}#1\end{align*}}
\newcommand{\eqn}[1]{\begin{align}#1\end{align}}

\newcommand{\hiddenproof}[2]{
\ifthenelse{#1}{Hidden}{#2}
}

\newcommand{\longcomm}[1]{}

\usepackage{datetime}
\newdateformat{findate}{\THEDAY.\THEMONTH.\THEYEAR}

\renewcommand{\pmat}[1]{\begin{bmatrix}#1\end{bmatrix}}
\renewcommand{\pmatsmall}[1]{\begin{bsmallmatrix}#1\end{bsmallmatrix}}

\newcommand{\yref}{y_{\mbox{\scriptsize\textit{ref}}}}

\renewcommand{\pmat}[1]{\begin{bmatrix}#1 \end{bmatrix}}

\newtheorem{theorem}{Theorem}[section]

\newtheorem{lem}[theorem]{Lemma}

\newtheorem{ass}[theorem]{Assumption}

\theoremstyle{definition}
\newtheorem{definition}[theorem]{Definition}
\newtheorem{rem}[theorem]{Remark}

\title[Finite-Dimensional Controllers for Robust Regulation ] 
      {Finite-Dimensional Controllers for Robust Regulation of Boundary Control Systems}

\author[Duy Phan and Lassi Paunonen]{}


 \email{duy.phan-duc@uibk.ac.at}
 \email{lassi.paunonen@tuni.fi}

\thanks{$^*$ Corresponding author: duy.phan-duc@uibk.ac.at}

\begin{document}
\maketitle

\centerline{\scshape Duy Phan$^*$}
\medskip
{\footnotesize
 \centerline{Institut f\"ur Mathematik, Leopold-Franzens-Universit\"at Innsbruck}
   \centerline{Technikerstra\ss e 13/7, A-6020 Innsbruck, Austria.}
} 

\medskip

\centerline{\scshape Lassi Paunonen}
\medskip
{\footnotesize
 \centerline{Mathematics, Faculty of Information Technology and Communication Sciences,}
   \centerline{Tampere University,}
   \centerline{PO. Box 692, 33101 Tampere, Finland.}
}

\bigskip

 \centerline{(Communicated by the associate editor name)}

\begin{abstract}
We study the robust output regulation of linear boundary control systems by constructing extended systems. The extended systems are established based on solving static differential equations under two new conditions. 
We first consider the abstract setting and present finite-dimensional reduced order controllers.  
The controller design is then used for particular PDE models: high-dimensional parabolic equations and beam equations with Kelvin-Voigt damping. Numerical examples will be presented using Finite Element Method. 
\end{abstract}

\maketitle

\section{Introduction}
We consider linear boundary control systems of the form \cite[Chapter 10]{TucWei09}
\begin{align*}
\dot{w}(t) &= \cA w(t), \qquad w(0) = w_0, \\
\cB w(t) &= u(t), \\
y(t) &= C_0 w(t)
\end{align*}
on a Hilbert space $X_0$ where $C_0$ is a bounded linear operator. 
The main aim of robust output regulation problem for boundary control systems is to design a dynamic error feedback controller so that the output $y(t)$ of the linear infinite-dimensional boundary control system  converges to a given reference signal $\yref (t)$, i.e. 
\begin{align*}
\| y(t) - \yref(t)\| \to 0, \quad \text{as~~} t \to \infty. 
\end{align*}
In addition, the control is required to be robust in the sense that the designed controller achieves the output tracking and disturbance rejection even under uncertainties and perturbations in the parameters of the system. 

\medskip
The robust output regulation and internal model based controller design for linear
  infinite-dimensional systems and PDEs --- with both distributed and boundary control --- has been considered in several articles, see~\cite{LogTow97,HamPohMMAR02,RebWei03,Imm07a,HamPoh10,Pau16a} and references therein.
In \cite{PauPhan19}, two finite-dimensional low-order robust controllers for parabolic control systems with distributed inputs and outputs were constructed. The main aim of this paper is to extend this design for linear boundary control systems. However, the main challenge is that the boundary input generally corresponds to an unbounded input operator. To tackle this issue, we construct an extended system with a new state variable $x = (v, u)^\top = (w - Eu, u)^\top$ where $E$ is an extension operator in such a way that the input operator of the new system is bounded. 

The construction of extension operator $E$ is one of key points of this paper. 
In the literature (for example \cite[Section 3.3]{CurZwa95}), the operator $E$ is chosen to be a right inverse operator of $\cB$. However, finding an arbitrary right inverse operator is not easy.
In this paper, we propose the additional conditions to construct the operator $E$. 
The construction of $E$ is completed by solving static differential equations. 
The idea comes from recent works on boundary stabilization for PDEs (for example \cite{Bad09,PhanRod18,Rod15}) or boundary control systems in abstract form (see \cite{Sal87,Sta05,TucWei09}) . 
Under our approach, the theory of partial differential equations guarantees the existence of the extension operator $E$. 
For simple cases (such as the heat equation with Neumann boundary control in Section \ref{sec-exHeatNeu}), the construction of $E$ by the new conditions does not give significant advantages compared to the choice of a right arbitrary inverse operator. 
Nevertheless, the advantage of our new approach can see clearly in more complicated partial differential equations (for example general linear parabolic equations on multi-dimensional domains, see the numerical example in Section \ref{sec-numex-para}). 
For these cases, the construction of right inverse operators by hand is not possible. In our approach we can approximate the operator $E$ by solving differential equations numerically and use the approximation in the controller design. 

For the reference signals, we assume that $\yref: \R \to \C^p$ can be written in the form
\begin{align}
\label{eq-refsig}
\yref(t) = a_0(t) + \sum_{k=1}^q \left(a_k(t) \cos(w_k t) + b_k (t) \sin (w_k t)\right) 
\end{align}
where all frequencies $\{ w_k \}_{k=0}^q \subset \R$ with $0 = w_0 < w_1 < \dots < w_q$ are known, but the coefficient polynomials vectors $\{a_k (t)\}_{k}$ and $\{b_k (t)\}_{k}$ with real or complex coefficients (any of the polynomials are allowed to be zero) are unknown. We assume the maximum degrees of the coefficient polynomial vectors are known, so that $a_k(t) \in \C^p$ are polynomial of order at most $n_k-1$ for each $k \in \{ 0,\dots,q\}$. 
The class of signals having the form \eqref{eq-refsig} is diverse. 
In Section \ref{sec-numex-para}, we present a numerical example with non-smooth reference signals. To track non-smooth signals, we approximate them by truncated Fourier series. In another numerical example, we track a signal where the coefficients are not constants. 

Under certain standing assumptions, we present an algorithm to design a robust controller for boundary control system by employing the finite-dimensional controllers in \cite{PauPhan19}. To apply the finite-dimensional controllers design for boundary control systems, we need some checkable assumptions to obtain the stabilizability and detectability of the extended systems. The assumptions can be influenced by free choices of some parameters in the construction of the extended systems. 
The next step is to utilize the controller design for two particular partial differential equations, namely linear diffusion-convection-reaction equations and linear beam equations with Kelvin-Voigt damping. For the case of beam equations, we present two different extended systems which work well both in theoretical and numerical aspects. 

The numerical computation is another contribution of this paper. Actually there are several numerical schemes satisfying the approximation assumption \ref{ass-A1} below. We also use Finite Element Method (FEM) as in \cite{PauPhan19} to simulate the controlled solution. We will present two numerical examples: a 2D diffusion-reaction-convection equation and a 1D beam equation with Kelvin-Voigt damping. In both examples, by choosing a suitable family of test functions, we approximate all operators and construct the extension operators $E$ numerically (in case we do not know $E$ explicitly). Then our finite-dimensional controllers can be computed through matrix computations. Another advantage of Finite Element Method is that this method can deal with various types of multi-dimensional domains (see the example in Section \ref{sec-numex-para}). 

\medskip
The paper is organized as follows. 
In Section \ref{sec-RORP}, we construct extended system from boundary control system with two additional assumptions on abstract boundary control systems, propose a collection of assumptions on the system, formulate the robust output regulation problem, and recall the Galerkin approximation. 
In Section \ref{sec-DesignCon}, we present the algorithm to design the robust controller for boundary control system and clarify that the controller solves the robust output regulation problem in Theorem \ref{the-RORP}. 
A block diagram of the algorithm for robust output regulation of boundary control systems will be presented in Section \ref{subsec-Algo}. 
Section \ref{sec-para} deals with general parabolic PDE models. Section \ref{sec-beam} concentrates on beam equations with Kelvin-Voigt damping. Two numerical examples will follow in each section by using Finite Element method.  

\subsection*{Notation}
For a linear operator $A: X \to Y$ we denote by $D(A),~\cN(A),~\cR(A)$ the domain, kernel, and range of $A$, respectively. 
$\rho(A)$ denotes the resolvent set of operator $A$, $\sigma(A) = \C \setminus \rho(A)$ denotes the spectrum of operator $A$. 
The space of bounded linear operators from $X$ to $Y$ is denoted by $\cL(X,Y)$.

\section{Boundary control systems and Robust Output Regulation}
\label{sec-RORP}

\subsection{Boundary control system} \label{sec-BcSysExtSys}
We start with the abstract boundary control system
\begin{subequations} \label{eq-abs-bc}
\begin{align}
\dot{w}(t) &= \cA w(t), \qquad w(0) = w_0, \\
\cB w(t) &= u(t), \\
y(t) &= C_0 w(t).
\end{align}
\end{subequations}
with $\cA: D(\cA) \subset X_0 \to X_0$, $u(t) \in U \coloneqq \C^m$, $y(t) \in Y \coloneqq \C^p$ and the boundary operator $\cB: D(\cA) \subset X_0 \to X_0 $. 

\begin{ass} \label{ass-boundsys1}
There exist two operators $\Ad$ and $\Arc$ satisfying $D(\Ad) = D(\cA) \subseteq D(\Arc)$ and the decomposition $\cA = \Ad + \Arc$, and $\Arc$ is relatively bounded with respect to $\Ad$. 
\end{ass}
$\Arc$ is relatively bounded to $A_d$ if $D(\Ad) \subseteq  D(\Arc)$  and there are non-negative constants $\alpha$ and $\beta$ so that
\begin{align*}
\| \Arc x \|  \le \alpha \|x\| + \beta \| \Ad x \| \quad
\text{for all~~} x \in D(\Ad). 
\end{align*}

The notations $\Ad$ and $\Arc$ are motivated by linear parabolic equations where we usually choose $\Ad$ as the diffusion term and $\Arc$ as the reaction-convection term.
 We assume that the system \eqref{eq-abs-bc} is a ``boundary control system'' in the sense of \cite{Sal87,TucWei09}. 

\begin{definition}
The control system \eqref{eq-abs-bc} is \emph{a boundary control system} if the followings hold: 

a. The operator $A_0: D(A_0) \to X_0$ with $D(A_0) = D(\cA) \cap  \ker(\cB)$ and  $A_0 x = \cA x$ for $x \in D(A_0)$ is the infinitesimal generator of a strongly continuous semigroup on $X_0$. 

b. $\cR (\cB) = U$. 
\end{definition}

The condition (b) implies that there exists an operator $E \in \cL(U,X_0)$ such that $\cB E = I$. However, finding an arbitrary right inverse operator of $\cB$ is not easy especially in the cases of multi-dimensional PDEs. Thus we propose the following additional assumption to construct the operator $E$. 

\begin{ass} \label{ass-boundsys2}
There exists a constant $\eta \ge 0$ such that $\eta \in \rho(A_0)$ and $E \in \cL(U,X_0)$ such that $\cR(E) \subset D(\cA)$ and
\begin{subequations} 
\begin{align}
\Ad E u &= \eta E u, \label{ass-cond1} \\
\cB E u &= u, \label{ass-condBC}
\end{align}
\end{subequations} 
for all $u \in U$.
\end{ass}
Under Assumptions \ref{ass-boundsys1} and \ref{ass-boundsys2}, $\Arc E$ is a bounded linear operator since $U$ is finite-dimensional and $\| \Arc E u \| \le \alpha \| Eu \| + \beta  \| \Ad Eu \| \le (\alpha + \beta \eta ) \|E\| \|u\|_U $. 

\begin{rem}
Comparing with the definition 3.3.2 in \cite{CurZwa95}, the condition \eqref{ass-cond1} is new. For particular PDEs, the construction of extension $E$ based on \eqref{ass-cond1} and \eqref{ass-condBC} leads to solve an ODE or an elliptic PDE. We call $E$ as ``an extension'' since its role is to transfer the boundary control into the whole domain. Note that the operator $E$ depends on the choice of $\eta \ge 0$. 
The approach of constructing an extension operator $E$ as a solution of an abstract elliptic equation has also been used, e.g., in  \cite{Bad09,PhanRod18,Rod15,Sal87}, \cite[Section 5.2]{Sta05}, and \cite[Remark 10.1.5]{TucWei09}).
\end{rem}

\subsubsection*{Assumptions on the system} We next introduce two assumptions on the system.

\begin{enumerate}
\renewcommand{\theenumi}{{\sf I\arabic{enumi}}} 
\renewcommand{\labelenumi}{} 
\item $\bullet$~Assumption~\theenumi:\label{ass-I1} The pair $(A_0,\,E)$ is exponentially stabilizable. 

\item $\bullet$~Assumption~\theenumi:\label{ass-I2} There exists $L_0\in \Lin(\C,X_0)$ such that $A_0+L_0 C_0$ is exponentially stable and 
  for every $k \in \{ 1, \dots, q\}$ we have 
$P_L(iw_k) \neq 0$ where $P_L(\gl)=C_0R(\gl,A_0+L_0C_0)E$.
\end{enumerate}
\medskip
Let $V_0$ be a Hilbert space, densely and continuously imbedded in $X_0$. We denote the inner product on $X_0$ and $V_0$ with $\langle\cdot,\cdot\rangle_{X_0}$ and $\langle\cdot,\cdot\rangle_{V_0}$, respectively. Analogously denote by $\|\cdot\|_{X_0}$ and $\|\cdot\|_{V_0}$ the norms on $X_0$ and $V_0$. 

\subsubsection*{Assumptions on the sesquilinear form} We assume that operator $A_0$ corresponds with sesquilinear $\sigma_0$ by the formula below 
\begin{align*}
\langle -A_0 w_1,\, w_2 \rangle = \sigma_0 (w_1,w_2), \qquad \forall w_1, w_2 \in V_0
\end{align*}
where $D(A_0) = \{ w \in V_0 \mid \sigma_0(w, \cdot) \text{~~has an extension to~~} X_0 \} $. The sesquilinear form $\sigma_0: V_0 \times V_0 \to \C$ satisfies two assumptions
\begin{enumerate}
\renewcommand{\theenumi}{{\sf S\arabic{enumi}}} 
\renewcommand{\labelenumi}{} 
\item $\bullet$~Assumption~\theenumi (Boundedness):\label{ass-S1} There exists $c_1 > 0$ such that for $w_1,~w_2 \in V_0$ we have 
\begin{align*}
| \sigma_0(w_1, w_2)| \le c_1 \|w_1\|_{V_0} \|w_2\|_{V_0}.  
\end{align*}

\item $\bullet$~Assumption~\theenumi (Coercivity):\label{ass-S2} There exist $c_2 > 0$ and some real $\lambda_0 > 0$ such that for $w \in V_0$, we have 
\begin{align*}
\re \sigma_0 (w, w) + \lambda_0 \|w\|^2_{X_0}  \ge c_2 \|w\|^2_{V_0}. 
\end{align*}
\end{enumerate} 
Under these assumptions, $A_0 - \lambda_0 I$ generates an analytic semigroup on $X_0$ (see \cite{BanKun84}). 

\subsection{Construction of the extended system}

By defining a new variable $v(t) = w(t) - E u (t)
$, we rewrite the equation \eqref{eq-abs-bc} in a new form 
\begin{subequations} \label{eq-abs-Cauchy}
\begin{align}
\dot{v}(t) &= A_0 v(t) - E(\dot{u}(t) - \eta u(t) )  + \Arc E u(t), \\
v(0) &= v_0. 
\end{align}
\end{subequations}
Since $A_0$ is the infinitesimal generator of an analytic semigroup, and $E,~\Arc E$ are bounded linear operators, Theorem 3.1.3 in \cite{CurZwa95} implies that the equation \eqref{eq-abs-Cauchy} has a unique classical solution for $v_0 \in D(A_0)$ and $u \in C^2([0, \tau]; U)$ for all $\tau >0$. The concept of ``classical solution'' means that $v(t)$ and $\dot{v}(t)$ are elements of $C((0,\tau), X_0)$ for all $\tau >0$, $v(t) \in D (A_0)$ and $v(t)$ satisfies  
\eqref{eq-abs-Cauchy}.
 \medskip

Denoting $\kappa(t) =  \dot{u}(t) -  \eta u(t)$, we obtain the extended systems with the new state variable $x = (v, u)^\top = \left(w - E u, u \right)^\top \in X \coloneqq X_0 \times U$ and a new control input $\kappa(t)$  as follows 
\begin{align} \label{eq-sys-ext}
\dot{x}(t) = 
\pmat{A_0 & \Arc E \\ 0 & \eta I} x(t)+ \pmat{-E \\ I} \kappa(t), \qquad x(0)= \pmat{w(0)-Eu(0) \\ u(0)}.
\end{align}
The observation part can be rewritten with the new variable as follows 
\begin{align} \label{eq-obs-ext}
y(t) = C_0 w(t) = C_0 \left(v(t) + E u(t) \right) = \pmat{C_0 & C_0 E } x(t). 
\end{align}

The theorem below shows the relationship between the solutions of \eqref{eq-abs-bc}, \eqref{eq-abs-Cauchy}, and \eqref{eq-sys-ext}. Its proof is analogous to the proof in \cite[Theorem 3.3.4]{CurZwa95}.  

\begin{theorem}
\label{the-change-var}
Consider the boundary control system \eqref{eq-abs-bc} and the abstract Cauchy equation \eqref{eq-abs-Cauchy}. Assume that $u \in C^2([0, \tau]; U)$ for all $\tau > 0 $. Then, if $v_0 = w_0 - E u(0) \in D(A_0)$, the classical solutions of \eqref{eq-abs-bc} and \eqref{eq-abs-Cauchy} are related by  
\begin{align*}
v(t) = w(t) - E u (t). 
\end{align*}
Furthermore, the classical solution of \eqref{eq-abs-bc} is unique. \\
In addition, if $v_0 \in D(A_0)$, the extended system \eqref{eq-sys-ext} with  $(x_0)_1 = v_0$, $(x_0)_2 = u(0)$ has the unique classical solution $x(t) = (v(t), u(t))^\top$, where $v(t)$ is the unique classical solution of \eqref{eq-abs-Cauchy}. 
\end{theorem}

\subsection{The Robust Output Regulation Problem}
We write the system \eqref{eq-sys-ext}-\eqref{eq-obs-ext} in an abstract form on a Hilbert space $X = X_0 \times U$.  
\begin{subequations}
\begin{align*}
\dot{x}(t) &= A x (t) + B \kappa(t), \\
y (t) &= C x(t)
\end{align*} 
\end{subequations}
where 
\begin{align} \label{eq-extABC}
A = \pmat{A_0 & \Arc E \\ 0 & \eta I}, \quad 
B = \pmat{-E \\ I}, \quad C = \pmat{C_0 & C_0 E}. 
\end{align}
Note that $B$ and $C$ are bounded operators. 
\medskip

We consider the design of internal model based error feedback controllers of the form on $Z = \C^{s}$ 
\begin{align*}
\dot{z}(t) &= \cG_1 z(t) + \cG_2 e(t), \quad z(0) = z_0 \in Z, \\
\kappa (t) &= K z(t), 
\end{align*}
where $e(t) = y(t) - \yref(t)$ is the regulation error, $\cG_1 \in \C^{s \times s}$, $\cG_2 \in \C^{s \times p}$, and $K \in \C^{m \times s}$. Letting $x_e (t) = (x(t), z(t))^\top$, the system and the controller can be written together as a closed-loop system on the Hilbert space $X_e = X \times Z$
\begin{align*}
\dot{x}_e (t) &= A_e x_e(t) + B_e \yref(t), \quad x_e (0) = x_{e0} \\
e(t) &= C_e x_e(t) + D_e \yref(t) 
\end{align*}
where $x_{e0} = (x_0, z_0)^\top$ and 
\begin{align*}
A_e  = \pmat{A & BK \\ \cG_2 C & \cG_1}, \quad B_e = \pmat{0 \\ -\cG_2}, \quad C_e = \pmat{C & 0}, \quad D_e = -I.   
\end{align*} 
The operator $A_e$ generates a strongly continuous semigroup $T_e(t)$ on $X_e$. 

\subsubsection*{The Robust Output Regulation Problem}
The matrices $(\cG_1, \cG_2, K)$ are to be chosen so that the conditions below are satisfied.  

(a) The semigroup $T_e(t)$ is exponentially stable. 

(b) There exists $M_e,~w_e >0$ such that for all initial states $x_0 \in X$ and $z_0 \in Z$ and for all signal $\yref(t)$ of the form \eqref{eq-refsig} we have 
\begin{align}
\label{eq-yerr}
\| y(t) - \yref(t) \| \le M_e e^{-w_e t} (\| x_{e0} \| + \| \Lambda \|).  
\end{align}
where $\Lambda$ is a vector containing the coefficients of the polynomials $\{a_k (t) \}_{k} $ and $\{b_k (t) \}_{k} $ in \eqref{eq-refsig}. 

(c) When $(A, B, C)$ are perturbed to $(\tilde{A}, \tilde{B}, \tilde{C})$ in such a way that the perturbed closed-loop system remains exponentially stable, then for all $x_0 \in X$ and $z_0 \in Z$ and for all signals $\yref(t)$ of the form \eqref{eq-refsig} the regulation error satisfies \eqref{eq-yerr} for some modified constants $\tilde{M}_e, \tilde{w}_e >0$. 

\subsection{Galerkin approximation}
\label{sec-GalApp}
Let $V_0^N \subset V_0$ be a sequence of finite-dimensional subspaces. We define $A^N_0:~V_0^N \to V_0^N$ by 
\begin{align*} 
\langle -A^N_0 v_1, v_2  \rangle = \sigma_0(v_1, v_2) \qquad \text{for all}\quad v_1,v_2 \in V^N_0, 
\end{align*}
that is, $A^N_0$ is defined via restriction of $\sigma_0$ to $V^N_0 \times V^N_0$. 
Assume that operator $\Arc$ corresponds with sesquilinear $\sigma_{\mathrm{rc}}$ by the formula  
\begin{align*}
\langle -\Arc w_1,\, w_2 \rangle = \sigma_{\mathrm{rc}} (w_1,w_2), \qquad \forall w_1, w_2 \in V_0
\end{align*}
where $D(\Arc) = \{ w \in V_0 \mid \sigma_{\mathrm{rc}}(w, \cdot) \text{~has an extension to~} X_0 \} $. We define  $\Arc^N:~V_0^N \to V_0^N$ by 
\begin{align*} 
\langle -\Arc^N v_1, v_2  \rangle = \sigma_{\mathrm{rc}}(v_1, v_2) \qquad \text{for all}\quad v_1,v_2 \in V^N_0, 
\end{align*}

\medskip

For a given $E \in \cL (U, X_0)$, we define $E^N \in \cL(U, V^N_0 )$ by 
\begin{align*} 
\langle E^N \kappa, v_2  \rangle = \langle \kappa, E^* v_2 \rangle_{X_0}  \qquad \text{for all~~} \quad v_2 \in V^N_0, 
\end{align*}
and $C_0^N \in \cL(V^N_0, Y)$ denotes the restriction of $C_0$ onto $V^N_0$. 
\medskip

Let $P^N$ denote the usual orthogonal projection of $X_0$ into $V^N_0$, i.e., for $v_1 \in V_0$
\begin{align*}
P^N v_1 \in V^N_0 \text{~~and~~} \langle P^N v_1, v_2 \rangle = \langle v_1, v_2 \rangle_{X_0}  \text{~~for all~~} v_2 \in V^N_0. 
\end{align*}
We assume an approximation assumption as follows 
 \begin{enumerate}
\renewcommand{\theenumi}{{\sf A\arabic{enumi}}} 
\renewcommand{\labelenumi}{} 

\item $\bullet$~Assumption~\theenumi:\label{ass-A1} For any $v \in V_0$, there exists a sequence $v^N \in V^N_0$ such that $\|v^N - v\|_{V_0} \to 0$ as $N \to \infty$. 
\end{enumerate}

\section{Reduced order finite-dimensional controllers}
\subsection{The controller}
\label{sec-DesignCon}

In this section, we recall a finite-dimensional controller design, namely ``Observer-based finite dimensional controller'' presented in \cite[Section III.A]{PauPhan19} to design robust controller for boundary control system \eqref{eq-abs-bc}. Another controller, namely ``Dual observer-based finite dimensional controller'' presented in \cite[Section III.B]{PauPhan19} can be applied analogously. 

\medskip
The finite-dimensional robust controller is based on an internal model with a reduced order observer of the original system and has the form
\begin{subequations}
  \label{eq:FinConObs}
  \eqn{
    \dot{z}_1(t)&= G_1z_1(t) + G_2 e(t)\\
    \dot{z}_2(t)&= (A_L^r+B_L^rK_2^r)z_2(t) + B_L^r K_1^N z_1(t) -L^r e(t)\\
    u(t)&= K_1^N z_1(t) + K_2^rz_2(t) 
  }
\end{subequations}
  with state $(z_1(t),z_2(t))\in Z:= Z_0\times \C^r$. 
All matrices 
$(G_1,G_2,A_L^r,B_L^r,K_1^N,K_2^r,L^r)$ are chosen based on the four-step algorithm given below.
The matrices $G_1,G_2$ are \keyterm{the internal model} in the controller. 
The remaining matrices $A_L^r,~B^r_L,~L^r,~K_1^N,~K_2^r$ are computed based on the Galerkin approximation $(A_0^N, \Arc^N, E^N, C_0^N)$ and model reduction of this approximation.

\medskip

\noindent \textbf{Step C1. The Internal Model:} \\
We choose $Z_0=Y^{n_0}\times Y^{2n_1} \times \ldots \times Y^{2n_q}$, $G_1 = \diag(J_0^Y, \ldots, J^Y_q)\in \cL(Z_0)$, and 
  $G_2=(G_2^k)_{k=0}^q \in \cL(Y,Z_0)$. 
  The components of $G_1$ and $G_2$ are chosen as follows.
  For $k=0$ we let
\begin{align*}
J^Y_0 = \pmat{
0_p  & I_p  	&  			&  \\
      	&  0_p & \ddots	&  \\
      	&			& \ddots	& I_p    \\
      	&			& 			& 0_p
} 
, \qquad 
G_2^0 = \pmat{0_p\\\vdots\\0_p\\I_p}
\end{align*}
where $0_p $ and $I_p$ are the $p\times p$ zero and identity matrices, respectively. 

 For $k \in \{ 1, \ldots, q \}$ we choose 
\begin{align*}
J^Y_k = \pmat{
  \Omega_k   & I_{2p}  &  &  \\
  &  \Omega_k  & \ddots&  \\
  && \ddots& I_{2p}    \\
  && & \Omega_k 
}
, \qquad 
G_2^k = \pmat{0_{2p}\\\vdots\\0_{2p}\\I_p\\0_p}
\end{align*}
where $\Omega_k = \pmatsmall{0_p&\gw_k I_p\\-\gw_k I_p&0_p}$.

\medskip

\noindent \textbf{Step C2. The Galerkin Approximation:}

For a fixed and sufficiently large $N\in\N$ we apply the Galerkin approximation described in Section~\ref{sec-GalApp} in $V_0$ to operators $(A_0, \Arc, E, C_0)$ to get their corresponding approximations $(A_0^N, \Arc^N, E^N, C_0^N)$. Then we compute the matrices $(A^N, B^N, C^N)$ as follows
\begin{align*}
A^N = \pmat{A_0^N & \Arc^N E^N \\ 0 & \eta I}, \quad
B^N =\pmat{-E^N \\ I}, \quad 
C^N = \pmat{C_0^N & C_0^N E^N}.
\end{align*}

\medskip
\noindent \textbf{Step C3. Stabilization:} \\
Denote the approximation $V^N \coloneqq V_0^N \times U$ of the space $V = V_0 \times U$.
Let $\ga_1,\ga_2\geq 0$.
Let $Q_1 \in \Lin(X,Y_0)$ and $Q_2 \in \Lin(U_0,X)$ with $U_0,Y_0$ Hilbert be such that $(A + \ga_2 I,Q_2)$ is exponentially stabilizable and $(Q_1,A+\ga_1 I)$ is exponentially detectable. 
 Let $Q_1^N$ and $Q_2^N$ be the approximations of $Q_1$ and $Q_2$, respectively. 
Let $Q_0\in \Lin(Z_0,\C^{p_0})$ be such that $(Q_0,G_1)$ is observable, and $R_1\in \Lin(Y)$ and $R_2\in \Lin(U)$ be such that $R_1>0$ and $R_2>0$.
We then define the matrices $(A^N_c, B^N_c,\,C^N_c)$ as follows
\eq{
  A^N_c = \pmat{G_1&G_2C^N\\0&A^N}, \quad B^N_c=\pmat{0\\B^N}, \quad C^N_c = \pmat{Q_0&0\\0&Q_1^N}.
}
Define
$L^N =-\Sigma_N C^N R_1\inv\in \cL(Y, V^N) $ 
and 
$K^N := \pmat{K_1^N,\; K_2^N} =-R_2\inv (B^N_c)^\ast\Pi_N \in \cL(Z_0\times V^N ,U )$ 
where
$\Sigma_N$
and
$\Pi_N$
are the non-negative solutions of finite-dimensional Riccati equations
\eq{
  (A^N + \ga_1 I) \Sigma_N + \Sigma_N (A^N + \ga_1 I)^*  - \Sigma_N \left(C^N \right)^* R_1\inv C^N \Sigma_N &=- Q_2^N (Q_2^N)^*, \\
  (A^N_c + \ga_2 I)^* \Pi_N + \Pi_N (A^N_c  +\ga_2 I) - \Pi_N B^N_c R_2\inv\left(B^N_c\right)^\ast \Pi_N &=- 
  \left(C^N_c\right)^\ast C^N_c.  
}

\medskip

\noindent \textbf{Step C4. The Model Reduction:}\\
For a fixed and suitably large $r\in\N$, $r\leq N$, by using the Balanced Truncation method to the stable finite-dimensional system
  \eq{
    (A^N+L^NC^N, [ B^N,\;L^N],K^N_2),
 } we obtain a stable $r$-dimensional reduced order system
\eq{
  \left(A_L^r,[B_L^r ,\; L^r] ,K_2^r\right).}

\medskip

The next theorem claims that the controller above solves Robust Output Regulation Problem for the boundary control systems \eqref{eq-abs-bc}. In Section \ref{sec-StabDetExtSys}, we present sufficient conditions for the stabilizablity and detectability of the extended system $(A,\,B,\, C)$ . 

\begin{theorem} \label{the-RORP}
Let assumptions \ref{ass-S1}, \ref{ass-S2}, \ref{ass-I1}, \ref{ass-I2}, and \ref{ass-A1} be satisfied. Assume that the extended system $(A,\,B,\, C)$ in \eqref{eq-extABC} is stabilizable and detectable. 
 The finite dimensional controller \eqref{eq:FinConObs} solves the Robust Output Regulation Problem provided that the order $N$ of the Galerkin approximation and the order $r$ of the model reduction are sufficiently high. 

If $\alpha_1, \alpha_2 >0$, the controller achieves a uniform stability margin in the sense that for any fixed $0 < \alpha < \min \{\alpha_1, \alpha_2\} $ the operator $A_e + \alpha I$ will generate an exponentially stable semigroup if $N$ and $r \le N $ are sufficiently large.   
\end{theorem}

\begin{proof}
The proof of this theorem is an application of Theorem III.2 in \cite{PauPhan19} under three checkable statements. 

\textbf{Step 1. ``Stabilizability and Detectability''} \\
 Recall the abstract system 
\begin{align*}
\dot{x}(t) &= A x (t) + B \kappa(t), \\
y (t) &= C x(t). 
\end{align*} 
We assume that the extended system $(A, B, C)$ is stabilizable and detectable. 
The sufficient conditions to guarantee the stabilizability and detectability of $(A, B, C)$ will be presented in Section \ref{sec-StabDetExtSys}. 

\medskip
\textbf{Step 2. ``Boundedness and Coercivity of the sesquilinear form''} \\
Define $V = V_0 \times U$ and $X = X_0 \times U$, the sesquilinear form $\sigma$ is defined by 
\begin{align} \label{sesqui-para}
\sigma (\phi_1, \phi_2) = \sigma \left((v_1, u_1), (v_2, u_2)  \right) =  \sigma_0 (v_1,v_2) - \langle \Arc E u_1, v_2 \rangle_{X_0} - \eta u_1 u_2^\top.  
\end{align}
For $\phi =(v,u)^\top \in V $, we define $\|\phi\|_X^2 = \|v\|_{X_0}^2 + \|u\|_{U}^2$ and $\|\phi\|_V^2 = \|v\|_{V_0}^2 + \|u\|_{U}^2$.

Since $\sigma_0$ satisfies two assumptions \ref{ass-S1} and \ref{ass-S2}, there exist constants $c_1 >0$, $c_2 >0$ and $\lambda_0 > 0$ such that for $v_1, v_2, \text{~and~} v \in V_0$ we have
\begin{align*}
|\sigma_0(v_1,v_2)| &\le c_1 \|v_1\|_{V_0} \|v_2\|_{V_0}, \\
\re \sigma_0 (v, v) + \lambda_0 \|v\|^2_{X_0} &\ge c_2 \|v\|^2_{V_0}. 
\end{align*} 
To check the boundedness of $\sigma(\phi_1, \phi_2)$, we have 
\begin{align*}
|\sigma(\phi_1, \phi_2)| &\le |\sigma_0(v_1, v_2)| + |\langle \Arc E u_1, v_2 \rangle_{X_0}| +  \eta  u_1 u_2^\top \\
&\le \left(c_1 + k  \|\Arc E\|_{\cL(U,X_0)} +  \eta \right)  \|\phi_1\|_{V} \|\phi_2\|_{V}. 
\end{align*} 
Regarding the coercivity of $\sigma(\phi, \phi)$, let $\phi = (v, u)$, we have 
\begin{align*}
\re \sigma(\phi, \phi) &= \re \sigma_0 (v,v) - \re \langle \Arc E u, v \rangle_{X_0} -  \eta  \|u\|^2_U \\
&\ge c_2 \left( \|v\|_{V_0}^2 + \|u\|^2_U \right) - \left( \lambda_0 + \frac{1}{2} \right) \|v\|^2_{X_0} \\
&\qquad - \left( \frac{1}{2}\|\Arc E\|_{\cL(U,X_0)}^2 + \eta + c_2  \right)\|u\|^2_U.  
\end{align*}
Define $\lambda_1 = \max \left\{ \lambda_0 + \frac{1}{2}, \frac{1}{2}\|\Arc E\|_{\cL(U,X_0)}^2 + \eta + c_2   \right\}$, we finally obtain $\re \sigma(\phi, \phi) + \lambda_1 \|\phi\|^2_X \ge c_2 \|\phi\|^2_V $. 

In conclusion the sesquilinear form $\sigma$ satisfies two assumptions \ref{ass-S1} and \ref{ass-S2} in the suitable spaces $X$ and $V$.  

\medskip
\textbf{Step 3. ``Approximation assumption''} \\
 Denote analogously $V^n = V_0^n \times U$. Under assumption \ref{ass-A1}, for any $v \in V_0$, there exists a sequence $v^n \in V_0^n$ such that $\|v^n - v\|_{V_0} \to 0$ as $n \to \infty$. Then for $x = (v, u) \in V$, define the sequence $x^n = (v^n, u) \in V^n$ satisfying $\|x^n - x\|_V \to 0$ as $n \to \infty$.  
\end{proof}

\subsection{Stabilizability and detectability of the extended systems}
\label{sec-StabDetExtSys}
In this section, we use three Theorems 5.2.6, 5.2.7, and  5.2.11 in \cite{CurZwa95}. We introduce new notations as follows. The spectrum of $A_0$ is decomposed into two distinct parts of the complex plane
\begin{align*}
\sigma^+(A_0) &= \sigma(A_0)\cap \overline{\C_0^+}, \quad \C_0^+  = \{ \lambda \in \C | \re \lambda  > 0 \}, \\
\sigma^-(A_0) &=  \sigma(A_0)\cap \C_0^-, \quad \C_0^- = \{ \lambda \in \C | \re \lambda  < 0 \}.
\end{align*}
Under the detectability of $(A_0, C_0)$ or the stabilizability of $(A_0, E)$, Theorems 5.2.6 or 5.2.7 in \cite{CurZwa95} guarantees that $A_0$ satisfies the spectrum decomposition assumption at 0. The decomposition of the spectrum induces a corresponding decomposition of the state space $X_0$, and of the operator $A$. We follow the definition of  $T_0^- (t)$ as in \cite[Equation 5.33]{CurZwa95}.

\begin{lem} \label{lem-detec-extsys}
Assume that $(A_0, C_0)$ is exponentially detectable. 

(i.) If $\Arc E = 0$ and $C_0 E$ is injective, the extended system $(A,C)$ is also exponentially detectable. 

(ii.) If $\Arc E \neq 0$, assume further that 
\begin{align} \label{eq-detec-cond}
\ker(\eta I - A)~\cap~\ker(C) = \{0\}
\end{align} then the extended system $(A, C)$ is exponentially detectable.  
\end{lem}
\begin{proof}
Since $(A_0, C_0)$ is exponentially detectable, Theorem 5.2.7 in \cite{CurZwa95} implies that $A_0$ satisfies the spectrum decomposition at 0, $T_0^- (t)$ is exponentially stable, and $\sigma^+(A_0)$ is finite. Then we can apply Theorem 5.2.11 in \cite{CurZwa95} for the detectable pair $(A_0, C_0)$ to obtain that 
\begin{align*}
\ker(sI - A_0) \cap \ker(C_0) = \{0\} \quad \text{for all~~} s \in \overline{\C}_0^+.  
\end{align*}
Under our choice $\eta \in \rho(A_0)$, the extended operator $A$ satisfies all conditions of Theorem 5.2.11. To prove the detectability of the extended system $(A,C)$, we will verify that 
\begin{align*}
\ker(sI - A) \cap \ker(C)= \{0\} \quad \text{for all~~} s \in \overline{\C}_0^+.  
\end{align*} 
Take $(v, u)^\top \in \ker(sI - A) \cap \ker(C)$, for any $s \in \overline{\C}_0^+$ we have
\begin{align} \label{eq-detec-gencon}
\begin{cases}
(sI - A_0) v - \Arc E u = 0, \\
(s - \eta) u = 0, \\
 C_0 v + C_0 E u = 0. 
\end{cases} 
\end{align}

(i.) If $\Arc E = 0$,  we rewrite the conditions \eqref{eq-detec-gencon} as $(sI - A_0) v = 0$, $(s - \eta) u = 0$, and $C_0 v + C_0 E u = 0$. 
For $s \in  \overline{\C}_0^+ \setminus  \eta$, we have that $u = 0$, $(sI - A_0)v = 0$, and $C_0 v = 0$. This implies that $v  \in \ker(sI - A_0) \cap \ker(C_0) = \{0\}$, and thus $v = 0$.\\
For $s = \eta \in \rho(A_0)$, we get that $v = 0$ and $C_0 E u = 0$. Under the condition that $C_0 E$ is injective, this implies that $u = 0$. 
 
Finally for all $s \in \overline{\C}_0^+$, we obtain that $\ker(sI - A) \cap \ker(C)= \{0\}$. It follows that the extended pair $(A, C)$ is exponentially detectable. 

(ii.) If $\Arc E \neq 0$, we first consider $s \in \overline{\C}_0^+ \setminus \eta$. Analogously as in the first case, we get that $u = 0$, $(sI - A_0)v = 0$, and $C_0 v = 0$. This implies that $v = 0$ due to the detectability of the pair $(A_0, C_0)$. 

In the case $s = \eta$, we rewrite the condition as 
$(\eta I - A_0) v - \Arc E u = 0$ and $ C_0 v + C_0 E u = 0$. Under the additional assumption \eqref{eq-detec-cond}, we get that $(v,u) = 0$. 

Since $\ker(sI - A) \cap \ker(C)= \{0\}$ for all $s \in \overline{\C}_0^+$, we conclude that the extended pair $(A, C)$ is exponentially detectable 
\end{proof}

\begin{lem} \label{lem-stab-extsys}
Assume that $(A_0, E)$ is exponentially stabilizable. 

(i.) If $\Arc E = 0$, the extended system $(A,B)$ is also exponentially stabilizable. 

(ii.) If $\Arc E \neq 0$, assume further that 
\begin{align} \label{eq-stab-cond}
\ker\left((sI - A_0)^\ast\right) \cap \ker \left(\left(\Arc E  +  (\eta - s)E \right)^\ast \right)  = \{0\} \qquad  \text{for~} s  \in \sigma^+(A_0),
\end{align}
the extended system $(A,B)$ is exponentially stabilizable. 
\end{lem}
\begin{proof}
The pair $(A_0, E)$ is exponentially stabilizable if and only if $(A_0^\ast, E^\ast)$ is exponentially detectable. Analogously as in Lemma \ref{lem-detec-extsys} we get that
\begin{align*}
\ker(sI - A_0^\ast) \cap \ker(E^\ast) = \{0\} \quad \text{for all~~} s \in \overline{\C}_0^+.  
\end{align*}
Since the pair $(A_0, E)$ is exponentially stabilizable, by Theorem 5.2.6 in  \cite{CurZwa95} the extended operator $A$ satisfies all conditions of Theorem 5.2.11 in \cite{CurZwa95}. To prove the stabilizability of the extended system $(A, B)$, we will check that 
\begin{align*}
\ker(sI - A^\ast) \cap \ker(B^\ast) = \{0\} \quad \text{for all~~} s \in \overline{\C}_0^+.  
\end{align*}
If $(v,u)^\top \in \ker(sI - A^\ast) \cap \ker(B^\ast)$, for $s \in \overline{\C}^+_0$ then 
\begin{align} \label{eq-stab-gencon}
\begin{cases}
(sI - A_0^\ast) v = 0, \\
 (-\Arc E)^\ast v + (s - \eta ) u = 0, \\ 
  -E^\ast v + u = 0. 
\end{cases}
\end{align}

(i.) If $\Arc E = 0$, the conditions \eqref{eq-stab-gencon} are rewritten as $(sI - A_0^\ast) v = 0$, $(s - \eta) u = 0$, $-E^\ast v + u = 0$. 
For $s \in \overline{\C}_0^+ \setminus \eta $, it follows that $u = 0$ and $(sI - A_0^\ast) v = 0$ and $E^\ast v = 0$. It is equivalent that $v \in \ker(sI - A^\ast) \cap \ker(E^\ast)$. Thus $v = 0$.
For $s = \eta$, since $\eta \in \rho(A_0)$, we get that $v = 0$. It follows that $u = 0$. 

Finally, for all $s \in \overline{\C}_0^+$, we get that $\ker(sI - A^\ast) \cap \ker(B^\ast) = \{0\}$. Therefore we conclude that the extended system $(A, B)$ is stabilizable. 

(ii.) We consider the case as $\Arc E \neq 0$. For $s \in \overline{\C}_0^+ \cap \rho(A_0^\ast)$, we get that $v=0$ and then $u = 0$. 

For $s \in  \sigma^+(A_0^\ast)$, we rewrite $u = E^\ast v$ and
\begin{align*}
0 &= (\Arc E)^\ast v - (s - \eta) u  = (\Arc E)^\ast v + (\eta  - s) E^\ast v  
= \left(\Arc E + (\eta - \bar{s} ) E \right)^\ast v 
\end{align*}
It follows that $v \in \ker \left( \left(\Arc E + (\eta - \bar{s} ) E \right)^\ast \right)$. Moreover $v \in  \ker \left((\bar{s} I - A_0)^\ast \right)$. Under the additional assumption \eqref{eq-stab-cond}, we get that $v = 0$, and then $u = 0$. 

In conclusion for all $s \in \overline{\C}_0^+$, we have that $\ker(sI - A^\ast) \cap \ker(B^\ast) = \{0\}$ and thus the extended system $(A,B)$ is stabilizable.  
\end{proof}
\begin{rem}
In \cite[Exercise 5.25]{CurZwa95}, we need  the assumption $0 \in \rho (A_0)$ to obtain the detectability and stabilizability of the extended systems $(A,B,C)$. In our approach, we instead require $\eta \in \rho(A_0)$. This condition is less restrictive since we can freely choose $\eta >0$ . 
\end{rem}

\begin{rem}
The additional conditions \eqref{eq-detec-cond} and \eqref{eq-stab-cond} to guarantee the detectability and stabilizability of the extended system in Lemmas \ref{lem-detec-extsys} and  \ref{lem-stab-extsys} are checkable. We need to check \eqref{eq-detec-cond} for only $\eta$ and  \eqref{eq-stab-cond} for finite $s \in \sigma^+(A_0)$. Under the Galerkin approximation, we can easily verify these conditions \eqref{eq-detec-cond} and \eqref{eq-stab-cond} by using the approximations of all operators. We then check these conditions below 
\begin{align*}
\ker\left(\eta I - A^N \right)~\cap~\ker\left(C^N\right) &= \{0\}, \\
\ker\left(\left(sI - A_0^N\right)^\ast\right) \cap \ker \left(\left(\Arc^N E^N  +  (\eta - s)E^N \right)^\ast \right)  &= \{0\} \qquad  \text{for~} s  \in \sigma^+\left(A_0^N\right). 
\end{align*}
\end{rem}

In the following, we present a block diagram of the algorithm for robust regulation of boundary control systems. 

\newpage

\subsection{The algorithm} \label{subsec-Algo}

\begin{center}
\begin{tikzpicture}[auto, thick, node distance=2.3cm, text width=10cm]
\node at (-5,1) [above=5mm, right=0mm] {\textsc{Extended system}};
\node[process] (Ext){
\textbf{Step E1. Extension $E$} \\
Construct an extension $E$ by solving a system \\
$\Ad E u = \eta E u,\quad \cB E u = u.$ 
};
\node (ConExt) [process, below of=Ext]{ 
\textbf{Step E2. Extended system $(A, B, C)$} \\
Construct an extended system $(A, B, C)$ where \\
$
A = \pmat{A_0 & \Arc E \\ 0 & \eta I}, \quad 
B = \pmat{-E \\ I}, \quad C = \pmat{C_0 & C_0 E}. 
$
};
 \draw[->]     (Ext) -- (ConExt);
     \draw[blue,thick,dotted] ($(Ext.north west)+(-0.2,0.5)$)  rectangle ($(ConExt.south east)+(0.2,-0.5)$);
    
\node at (-5,-5) [above=8mm, right=0mm] {\textsc{The controller}};
 \node (InMod)   [process, below of=ConExt, yshift = -0.9cm ]          {
\textbf{Step C1. The Internal Model} \\  
Choose $G_1$ and $G_2$ incorporating the internal model.
  };
  \node (Gar) [process, below of=InMod]   {
 \textbf{Step C2. The Galerkin Approximation} \\
Fix $N \in \N$, apply the Galerkin approximation to \\
operators $(A_0, \Arc, E, C_0)$  to get their corresponding \\ matrices $(A_0^N, \Arc^N, E^N, C_0^N)$.  Then compute the matrices $(A^N, B^N, C^N)$ as the approximations of $(A,B,C)$. 
};

\node[process, below of=Gar] (Stab){
 \textbf{Step C3. Stabilization} \\
Choose $L^N, K_1^N, K_2^N$ by solving finite-dimensional Riccati equations with the matrices $(A^N, B^N, C^N)$ and $(G_1, G_2)$.
};
  \node (ModRed) [process, below of=Stab]   {
   \textbf{Step C3. The Model Reduction} \\
Fix $r \le N$, use Balanced Truncation Method to get a stable $r-$dimensional system $(A_L^r,[B_L^r ,\; L^r] ,K_2^r)$
  };
  \draw[->]     (InMod) -- (Gar);
  \draw[->]		(Gar) -- (Stab); 
  \draw[->]     (Stab) -- (ModRed);
  \draw[->]    (ConExt) -- (InMod);
  \draw[red,thick,dotted] ($(InMod.north west)+(-0.2,0.5)$)  rectangle ($(ModRed.south east)+(0.2,-0.5)$);
\end{tikzpicture}
\end{center}

\section{Boundary control of parabolic partial differential equations} \label{sec-para}
We consider controlled parabolic equations with Dirichlet boundary controls, for time $t >0$,  in a $C^\infty\text{-smooth}$ domain $\Omega \subset \R^d$ with $d$ a positive integer, located locally on one side of its boundary $\partial \Omega = \Gamma_c \cup \Gamma_u,~~\Gamma_c \cap \Gamma_u = \emptyset$ as follows  
\begin{subequations} \label{eq-para}
\begin{align}
\frac{\partial w}{\partial t}(\xi, t) - \nu \Delta w (\xi, t)  + \alpha(\xi) w(\xi, t)+ \nabla \cdot \left( \beta (\xi) w (\xi, t) \right) &= 0,  \quad w(x,0) = w_0(\xi),\\
w (\bar{\xi}, t) = \sum \limits_{i = 1}^m {u_i(t)  \psi_i(\bar{\xi})} \text{~~for~~} \bar{\xi}\in \Gamma_c,\quad w (\bar{\xi}, t)  &= 0  \text{~~for~~} \bar{\xi}\in \Gamma_u.  
\end{align}
\end{subequations} 
In the variable $(\xi, \bar{\xi}, t)\in \Omega \times \Gamma \times (0, +\infty)$, the unknown in the equation is the function $w = w(\xi,t) \in \R$. The diffusion coefficient $\nu$ is a positive constant. The functions $\alpha : \R \to \R$ and $\beta: \R^d \to \R$ are fixed and depend only on $\xi$. Function $w_0$ is known. We also assume that $\alpha \in L^\infty (\Omega, \R)$ and $\beta \in L^\infty (\Omega, \R^d)$. 

The functions $ \psi_i (\bar{\xi})$ are fixed and will play the role of boundary actuators. The control input is $u(t) = (u_i(t))_{i=1}^m \in U = \C^m$ (see \cite{PhanRod18} and example below). 

Analogously we assume the system has $p$ measured outputs so that \\ $y(t) = (y_k(t))_{k=1}^p \in Y = \R^p$ and
\begin{align*}
y_k(t) = \int_\Omega {w(\xi,t) c_k (\xi) d\xi}, 
\end{align*}
for some fixed $c_k(\cdot) \in L^2(\Omega,\R)$. The output operator $C_0 \in \cL(X_0,Y)$ is such that $C_0 w  =  \left( \langle w, c_k \rangle_{L^2(\Omega)} \right)_{k=1}^p$ for all $w \in X_0$.

\subsection{Constructing the extended system}
We choose $X_0 = L^2(\Omega, \R)$, $V_0 = H_0^1 (\Omega, \R)$ and denote $X = X_0 \times U,~~V = V_0 \times U$. Denote $v = w - Eu$, $\Ad w \coloneqq \nu \Delta w$ and   $\Arc w \coloneqq -\alpha w - \nabla \cdot (\beta w)$. \\
For each actuator $\psi_i \in H^{\frac{3}{2}}(\Gamma)$, we choose the extension $\Psi_i \in H^2(\Omega)$ which solves the elliptic equations 
\begin{align}\label{eq-ext-elliptic}
\nu \Delta \Psi_i = \eta \Psi_i, \qquad \Psi_i\mid_{\Gamma_c} = \psi_i, \qquad \Psi_i\mid_{\Gamma_u} = 0.
\end{align}
We then set the operator $E: U \to S_\Psi$ with $S_\Psi \coloneqq \Span \{ \Psi_i \mid i \in \{1,2,\dots, m \} \} $ as
\begin{align*}
Eu \coloneqq \sum_{i=1}^m u_i \Psi_i. 
\end{align*}
 
We rewrite the boundary control problem with the new state variable $x = (v, u)^\top = (w-Eu, u)^\top$. 
The new dynamic control variable $\kappa(t) \in U$ is defined as $\kappa_i (t) = \dot{u}_i(t) - \eta  u_i(t)$ for all $i \in \{1, \dots, m \}$. The new input operator $B \in \cL(U, X)$ is such that $B \kappa = \pmat {-E \\ I} \kappa = \pmat {-\sum_{i=1}^m \kappa_i \Psi_i  \\ \kappa}$  for all $\kappa \in U$. The new output operator $C \in \cL (X, Y)$ is such that 
\begin{align*}
C x = \left(\int_\Omega v(\xi)c_k(\xi) d\xi + \sum_{i=1}^m u_i \int_\Omega  \Psi_i (\xi) c_k (\xi) d\xi \right)_{k=1}^p
\end{align*}
for all $x \in X$. We get an extended system with $(A,B,C)$ as in \eqref{eq-extABC}.

As shown in \cite[Section V. B.]{PauPhan19}, the sesquilinear $\sigma_0$ corresponding with operator $A_0$ is bounded and coercive. Thus the sesquilinear form $\sigma$ corresponding with the extended operator $A$ here has the same properties (as shown in the proof of \eqref{the-RORP}).

\subsection{A 1D heat equation with Neumann boundary control} \label{sec-exHeatNeu}
In this section we consider a 1D heat equation with Neumann boundary control and construct the extended system by our approach. Reformulating this control system as an extended system was also considered in \cite[Example 3.3.5]{CurZwa95} with the choice of right inverse operator. We first introduce the PDE model
\begin{align*}
 \frac{\partial w}{\partial t} (\xi,t) &=  \frac{\partial^2 w}{\partial \xi^2} (\xi,t), \quad  \frac{\partial w}{\partial \xi} (0,t) = 0,~~\frac{\partial w}{\partial \xi} (1,t) = u(t), \\
w(\xi,0) &= w_0(\xi). 
\end{align*}
To construct the extended system, we define $X_0 = L^2(0,1),~U = \C$. The operator $\cA =\frac{\partial^2}{\partial \xi^2}$ is with domain 
$
D(\cA) = \big\{ h \in H^2(0,1) \mid \frac{dh}{d\xi} (0) = 0 \big\}
$
and the boundary operator $\cB h = \frac{dh}{d\xi} (1)$ with $D(\cB) = D(\cA)$.

We define operator $A_0 = \frac{d^2}{d \xi^2}$ with domain 
\begin{align*}
D(A_0) = D(\cA) \cap \ker(\cB) = \Big\{ h \in H^2(0,1) \mid \frac{dh}{d\xi} (0) = \frac{dh}{d\xi} (1) = 0 \Big\}.
\end{align*}
By choosing $\Ad = A_0$, we define $Eu(t) = g(\xi) u(t)$ where $g(\xi)$ solves the following second order ODE
\begin{align*}
g'' (\xi) = g(\xi),\quad g'(0) = 0,\quad g'(1) = 1.  
\end{align*}
By solving this ODE, we get $g(\xi) = \frac{2 \mathrm{e}}{\mathrm{e}^2-1} \cosh \xi$. By denoting an extended variable $x = (v, u)^\top$, we get an abstract system $\dot{x}(t) = A x(t) + B\kappa(t)$ where 
\begin{align*}
A = \pmat{\frac{d^2}{d\xi^2} & 0 \\0 & 1}, \quad B = \pmat{-E \\ 1}. 
\end{align*}
In \cite{CurZwa95}, the function $g(\xi) = \frac{1}{2}\xi^2$ was chosen and also defined $Eu(t) = g(\xi) u(t)$. This choice leads to another extended system with
$
A = \pmat{\frac{d^2}{d\xi^2} & 1 \\0 & 0}
$ and the same $B$. We emphasize that the corresponding operator $A$ does not coincide with the choice of $\eta = 0$ in our approach.
 We then apply the controller design in Section \ref{sec-DesignCon} for the extended systems. 

For simple PDE models, the construction of extension $E$ using our approach does not yet give significant advantages over the method presented in \cite{CurZwa95}. We will next present a two-dimensional PDE model where the construction of $E$ would not be possible by hand. 

\subsection{A 2D diffusion-reaction-convection model}
\label{sec-numex-para}
In this example, we consider the equation \eqref{eq-para} on a domain $\Omega = \left(\bigcup\limits_{i = 1}^6 \Omega_i \right) \setminus \Omega_7$ (plotted in \eqref{fig_domain2D}) where 
\begin{align*}
\Omega_1 &= \{ (\xi_1, \xi_2) \in \R^2  \mid -2 < \xi_1 < 0,~~ -1 < \xi_2 \le 1  \}, \\ 
\Omega_2 &= \{ (\xi_1, \xi_2)\in \R^2 \mid \xi_1^2 + \xi_2^2 < 1,~~ \xi_1 \ge 0,~~\xi_2 \le 0  \} , \\
\Omega_3 &= \{ (\xi_1, \xi_2) \in \R^2 \mid -1 \le \xi_1 \le 1,~~0 < \xi_2 < 2 \}, \\
\Omega_4 &= \{ (\xi_1, \xi_2) \in \R^2 \mid \xi_1^2 + (\xi_2 - 2)^2 < 1,~~  \xi_2 \ge 2  \}, \\
\Omega_5 &= \{ (\xi_1, \xi_2) \in \R^2 \mid (\xi_1+2)^2 + (\xi_2 - 2)^2 > 1,~~ -2 < \xi_1 \le -1,~~1 \le \xi_2 < 2 \}, \\
\Omega_6 &= \{ (\xi_1, \xi_2) \in \R^2 \mid (\xi_1 + 2)^2 + \xi_2^2 < 1,~~\xi_1 \le -2  \} \\
\Omega_7 &= \left\{ (\xi_1, \xi_2) \in \R^2 \mid \left(\xi_1 + \frac{3}{2} \right)^2 + \left(\xi_2 - \frac{1}{4} \right)^2 \le \frac{4}{25} \right \}. 
\end{align*}
The boundary $\Gamma$ can be described as seven segments
\begin{align*}
\Gamma_1 &= \{ (\xi_1,-1 ) \in \R^2 \mid -2 < \xi_1 < 0 \},\\
\Gamma_2 &= \{ (\xi_1, \xi_2)\in \R^2 \mid \xi_1^2 + \xi_2^2 = 1,~~ \xi_1 \ge 0,~~\xi_2 \le 0  \},\\
\Gamma_3 &= \{ (1, \xi_2) \in \R^2 \mid 0 < \xi_2 < 2 \} \\
\Gamma_4 &= \{ (\xi_1, \xi_2)\in \R^2 \mid  \xi_1^2 + (\xi_2 - 2)^2 = 1,~~\xi_2  \ge 2 \}, \\
\Gamma_5 &= \{ (\xi_1, \xi_2)\in \R^2 \mid (\xi_1+2)^2 + (\xi_2 - 2)^2 = 1,~~\xi_1 > -2,~~\xi_2 < 2 \}, \\
\Gamma_6 &= \{ (\xi_1, \xi_2)\in \R^2 \mid (\xi_1 + 2)^2 + \xi_2^2 = 1,~~\xi_1 \le -2    \} \\
\Gamma_7 &= \left\{ (\xi_1, \xi_2)\in \R^2 \mid   \left(\xi_1 + \frac{3}{2} \right)^2 + \left(\xi_2 - \frac{1}{4} \right)^2 = \frac{4}{25} \right\}. 
\end{align*}

\begin{figure}[]  
\centering
\includegraphics[width=.5\textwidth]{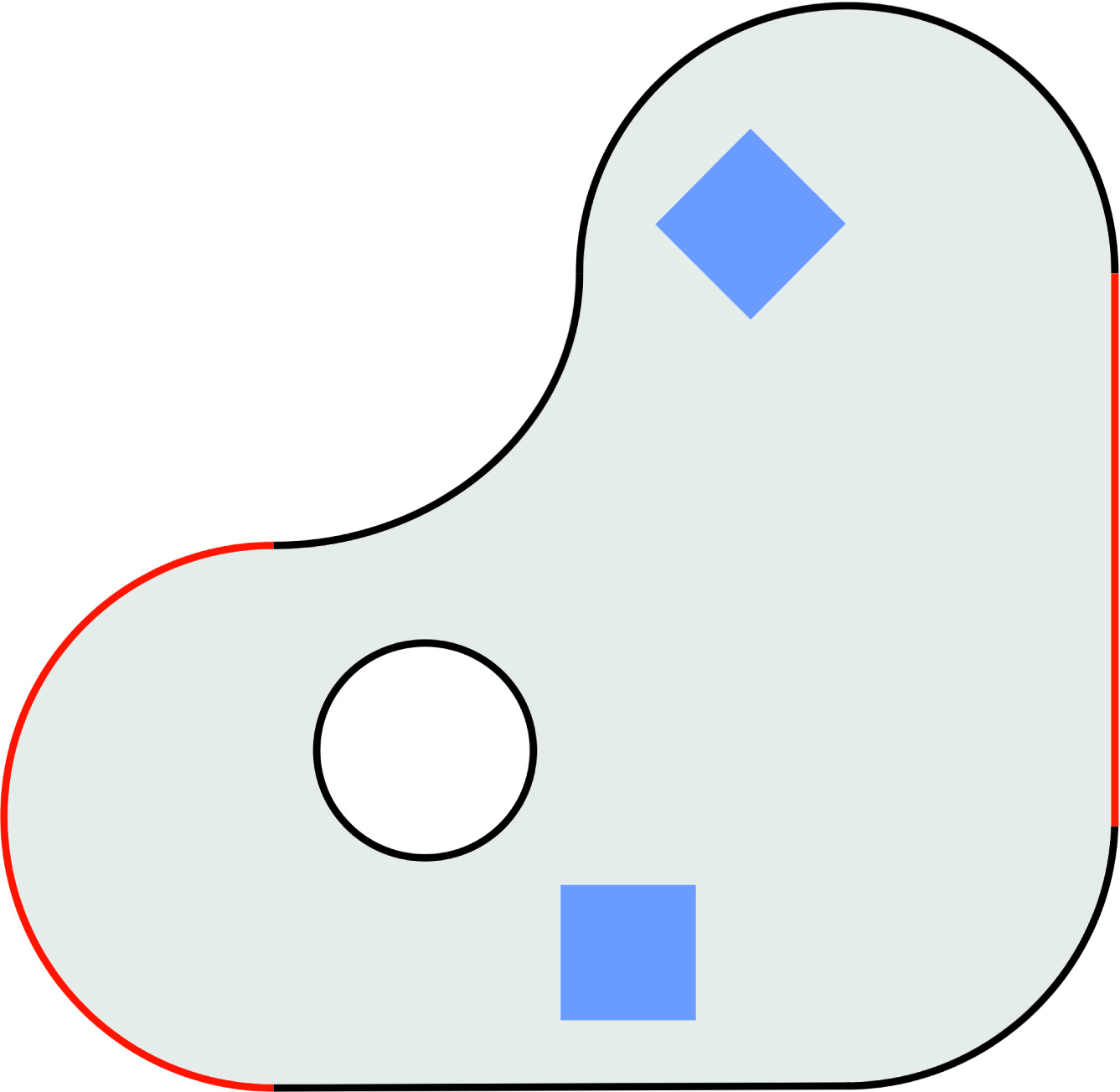}
\caption{Boundary controls located on red segments and regions of observations (blue).}
\label{fig_domain2D}
\end{figure}

We take $\nu = 0.5$, $\alpha (\xi) = 3 (\xi_1 + \xi_2),~\beta_1(\xi)= \cos (\xi_1) - \sin (2 \xi_2) - 2,~\beta_2 (\xi) = \sin (3 \xi_1) + \cos (4 \xi_2),~\beta = (\beta_1, \beta_2)$ in \eqref{eq-para}.

We consider \eqref{eq-para} with two boundary inputs located in two distinct segments $\Gamma_3$ and $\Gamma_6$ (see red segments of boundary in Figure \ref{fig_domain2D}), i.e. $\Gamma_c = \Gamma_3 \cup \Gamma_6$ and $\Gamma_c = \Gamma_1 \cup \Gamma_2 \cup \Gamma_4 \cup \Gamma_5 \cup \Gamma_7$  . On these segments, for $\bar{\xi} = \left( \bar{\xi_1}, \bar{\xi_2} \right) \in \Gamma$, we take $\psi_1 (\bar{\xi}) = \sin \left( \frac{\pi \bar{\xi_2}}{2}\right) \chi_{\Gamma_3}(\bar{\xi})$ and $\psi_2 (\bar{\xi}) = \sin \left( 3\left(\theta(\bar{\xi}) - \frac{\pi}{2} \right) \right) \chi_{\Gamma_6}(\bar{\xi})$ where $\theta(\bar{\xi}) = \arctan2 \left(  \frac{\bar{\xi_1}+2}{\bar{\xi_2}} \right) $. Next we define two extensions of boundary controls by solving elliptic equations with $\eta = \nu = 0.5$ 
\begin{align} 
\nu \Delta \Psi_i = \eta \Psi_i, \qquad \Psi_i\mid_{\Gamma_c} = \psi_i, \quad \Psi_i\mid_{\Gamma_u} = 0 \quad i \in \{1,~2\}. 
\end{align}
Two corresponding solutions $\Psi_1$ and $\Psi_2$ are plotted in Figure \ref{fig_para_ext}. 

\begin{figure}[]
\centering
\subfloat[Extension of $\psi_1$.]{\includegraphics[width=.49\textwidth]{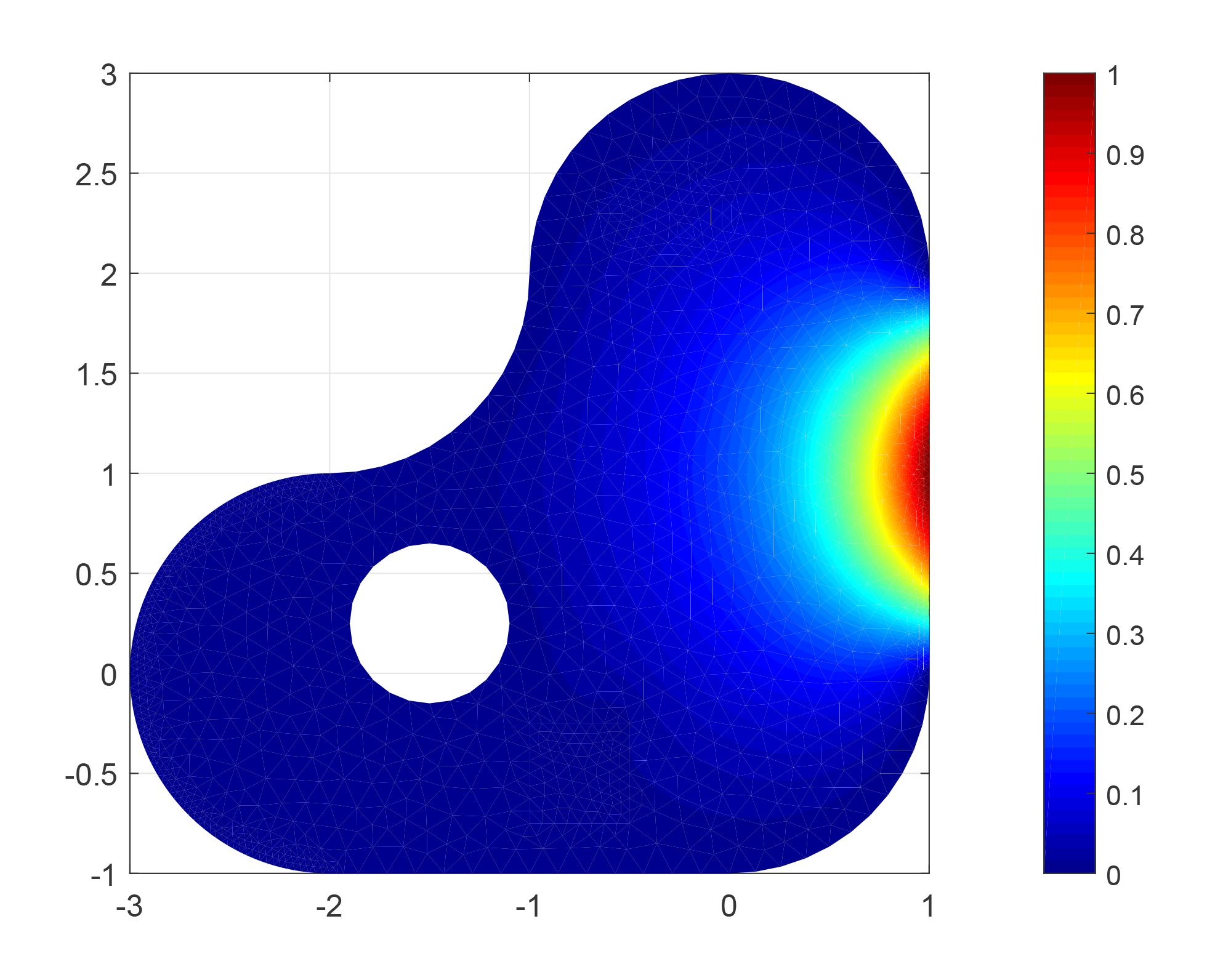}}
\subfloat[Extension of $\psi_2$.]{\includegraphics[width=.47\textwidth]{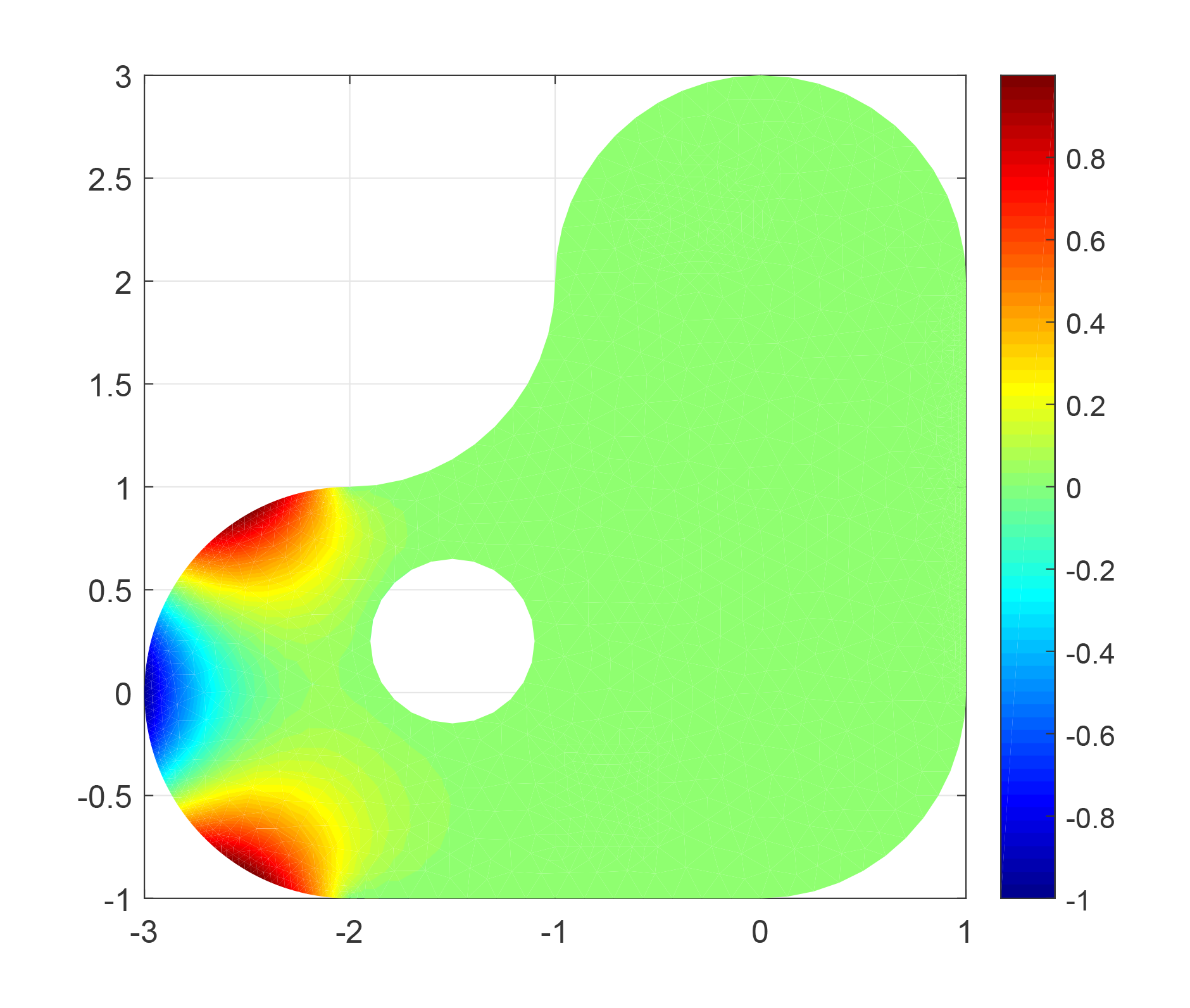}}
\caption{Two extensions of boundary actuators.}
\label{fig_para_ext}
\end{figure}

\begin{rem}
In the theoretical results, we have asked the boundary actuators to be in $H^\frac{3}{2}(\Gamma)$.  Two actuators $\psi_1 (\bar{\xi})$ and $\psi_2 (\bar{\xi})$ above are actually in $H^s(\Gamma)$ with $s < \frac{3}{2}$, but not necessarily in $H^\frac{3}{2}(\Gamma)$. This lack of regularity will be neglected in simulation. 
\end{rem}

Two measurements act on blue rectangular subdomains of $\Omega$ (see Figure \ref{fig_domain2D}). The rectangular $\Omega_{m1}$ has four corners 
$$(-1,~-.75),~~(-.5,~-.75),~~(-.5,~-.25),~~(-1,~-.25),$$ 
and the rectangular $\Omega_{m2}$ has four corners 
$$(-.3,~1.9464),~~(.0536,~2.3),~~(-.3,~2.6536),~~(0.6536,~2.3).$$
More precisely, we choose 
$c_1(\cdot) = \chi_{\Omega_{m1}} (\cdot) ,~~ c_2(\cdot) = \chi_{\Omega_{m2}} (\cdot)$. 
Our aim is to track a non-smooth periodic reference signal $\yref(t+2) = \yref(t) = (y_1(t), y_2(t)),~~\forall t \ge 0$ where 
\begin{align*}
y_1(t) = \begin{cases}
1  \qquad &\text{if~~} 0 \le t < \frac{1}{2}, \\ 
-2 t + 2 \qquad &\text{if~~} \frac{1}{2} \le   t < 1, \\
0 \qquad &\text{if~~} 1 \le   t < \frac{3}{2}, \\
2 t -3 \qquad &\text{if~~} \frac{3}{2} \le   t < {2}, 
\end{cases} 
\end{align*}
and 
\begin{align*}
y_2(t) = \begin{cases}
-t  \qquad &\text{if~~} 0 \le t < 1, \\ 
t-2 \qquad &\text{if~~} 1 \le   t < 2. \\ 
\end{cases} 
\end{align*}
This type of signals is approximated by truncated Fourier series
\begin{align*}
\yref(t) \approx a_0 (t) + \sum \limits_{k=1}^q (a_k (t) \cos ( k \pi t ) + b_k (t) \sin ( k \pi t )).
\end{align*}
Here we use $q = 10$ and the corresponding set of frequencies is $\{ k \pi \mid  k \in \{ 0, 1, \dots, 10 \} \}$ and $n_k = 1$ for all $k \in \{ 0, 1, \dots, 10 \}$. 
The domain $\Omega$ is approximated by a polygonal domain $\Omega_D$ and we consider a partition of $\Omega_D$ into non-overlapping triangles to discretize the extended system using Finite Element Method.  
We construct the observer-based controller using a Galerkin approximation with order $N = 1956$ and subsequent Balanced Truncation with order $r = 30$. The internal model has dimension $\dim Z_0 =  2 \times 2  \times 10 + 2 \times 1  \times 1 = 42$. The parameters of the stabilization are chosen as 
\begin{align*}
\alpha_1 = 0.65,~~ \alpha_2 = 0.95,~~R_1 = I_2,~~R_2 = 10^{-2}I_2. 
\end{align*}
The operators $Q_0,~ Q_1$, and $Q_2 $ are freely chosen such that $Q_2 Q_2^\ast = I_X$ and $C_c^\ast C_c = I_{Z_0 \times X}$.
Another Finite Element approximation with $M = 2688$ is constructed to simulate the original system. The initial states to solve the controlled system are $v_0 (\xi) = 0.25 \sin (\xi_1)$ and $u_0 = 0 \in \R^{42 + 30}$. The tracking signals are plotted in Figure \ref{fig_track2D}. In Figure \ref{fig_hsv2D}, the first Hankel singular values of the Galerkin approximation are plotted.

\begin{figure}[] \centering 
\includegraphics[width=.8\textwidth]{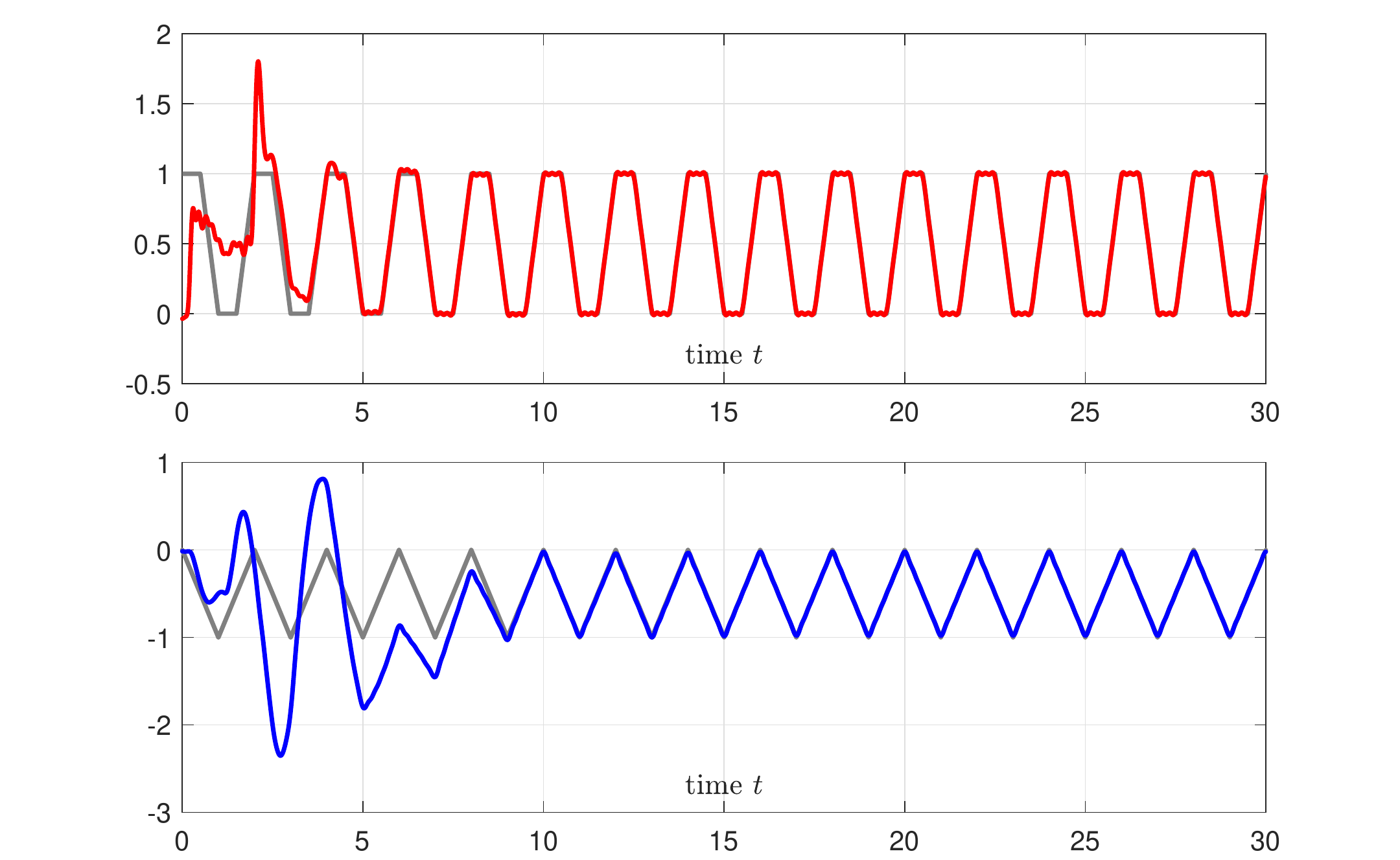}
\caption{Output tracking of the boundary control of the 2D parabolic equation.}
\label{fig_track2D}
\end{figure}

\begin{figure} \centering 
\includegraphics[width=.8\textwidth]{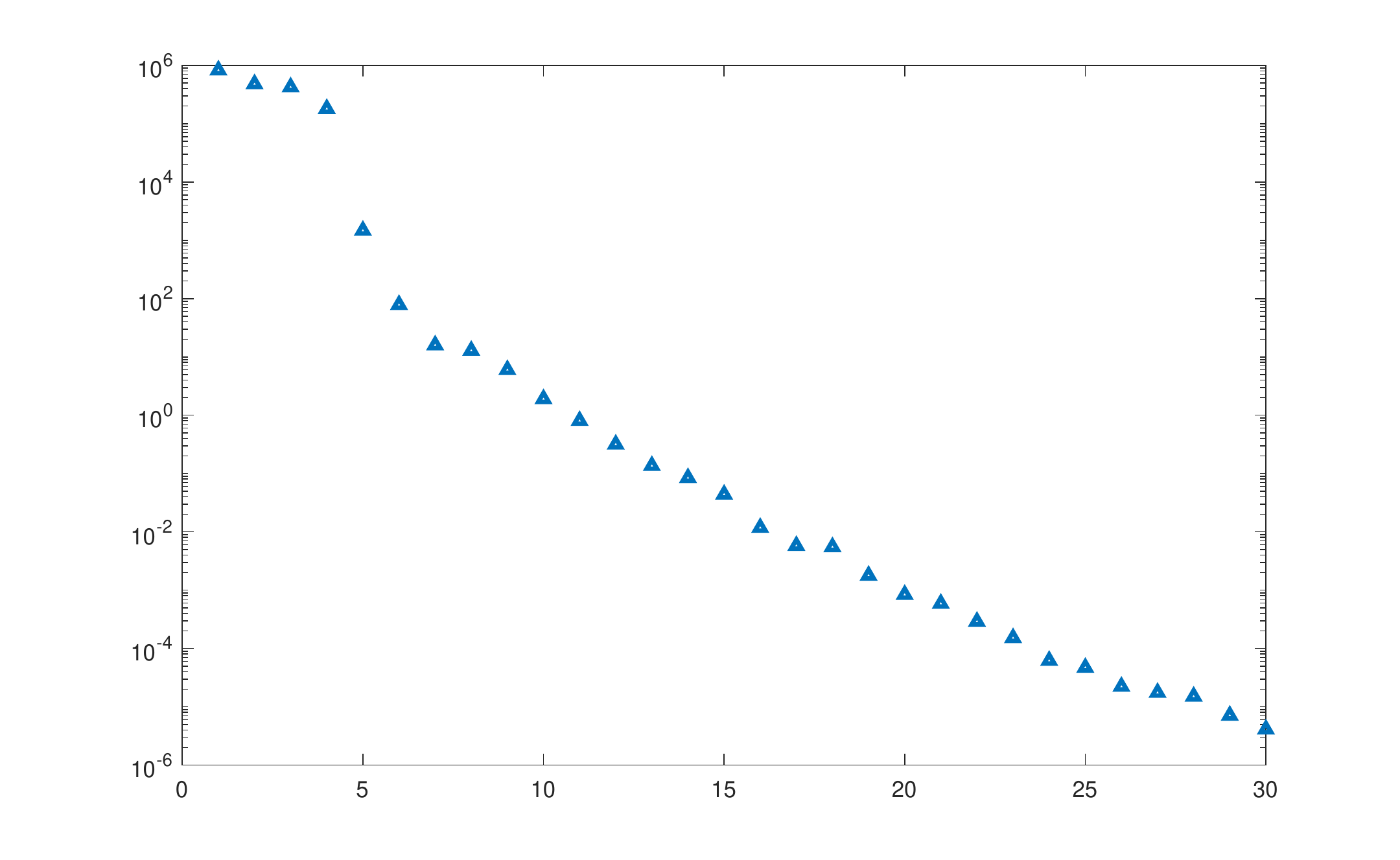}
\caption{Hankel singular values.}
\label{fig_hsv2D}
\end{figure}

\section{Boundary control of a beam equation with Kelvin-Voigt damping} \label{sec-beam}
Consider a one-dimensional Euler-Bernoulli beam model on $\Omega = (0,l)$
\begin{subequations} \label{eq-beam}
\begin{align}
\frac{\partial^2 w}{\partial t^2} (\xi,t) &+ \frac{\partial^2}{\partial \xi^2} \left( \alpha \frac{\partial^2 w}{\partial \xi^2} (\xi,t) + \beta \frac{\partial^3 w}{\partial \xi^2 \partial t} (\xi,t)  \right) + \gamma \frac{\partial w}{\partial t} (\xi,t) = 0, 
\\ 
w(\xi,0) &= w_0 (\xi), \qquad \frac{\partial w}{\partial t}(\xi,0) = w_1 (\xi), \\
y(t) &= C_1 w(\cdot, t) + C_2 \dot{w}(\cdot, t), 
\end{align}
\end{subequations}
with the  constants  $\alpha, \beta > 0$ and $\gamma \ge 0$. The measurement operators for the deflection $w(\cdot,t)$ and the velocity $\dot{w}(\cdot, t)$ are such that $C_j w  = \left(\langle w, c^j_k \rangle_{L^2} \right)_{k=1}^p \in Y = \R^p$ for $w \in L^2(0,l)$ and $j = 1,2$ for some fixed functions $c^j_k(\cdot) \in L^2 (0,l)$.
We consider boundary conditions 
\begin{align*}
w(0,t) = \frac{\partial w}{\partial \xi} (0,t) =  \frac{\partial^3 w}{\partial \xi^3} (l,t) =  0, \quad \frac{\partial^2 w}{\partial \xi^2} (l,t) = u(t),
\end{align*}
where $u(t)$ is the boundary input at $\xi = l$. This type of boundary controls was considered in \cite[Section 10.4]{TucWei09} and \cite{Guo14} (with boundary disturbance signals).  

Let $W_0 = \left\{ w \in H^2(0, l) \mid   w(0) = \frac{dw}{d\xi}(0) = 0  \right\}$ and define the inner product on $W_0$ by 
\begin{align*}
\langle w_1 , w_2 \rangle_{W_0}  = \int_0^l w_1 (\xi) w_2(\xi) d\xi, \quad \forall w_1, w_2 \in W_0. 
\end{align*}
We define the spaces $X_0 = W_0 \times L^2 (0,l) $, $V_0 = W_0 \times W_0$ and the operator 
\begin{align*}
\cA &= \pmat{0 & \bI \\ -\alpha \frac{\partial^4}{\partial \xi^4} & -\beta \frac{\partial^4}{\partial \xi^4} - \gamma}, \\
\text{with~} D(\cA) &= \left \{ (w_1, w_2) \in V_0 \mid \alpha \frac{d^2}{d\xi^2} w_1 + \beta \frac{d^2}{d\xi^2} w_2 \in H^2(0,l), ~~ \frac{d^3 w_1}{d\xi^3}(l) = 0    \right \}.  
\end{align*}
The boundary operator $\cB:  X_0 \to \C$ denotes by $\cB \pmat{w_1 \\ w_2} = \frac{d^2 w_1}{d\xi^2} (l)$.\\
The operator $A_0$ is given by 
$
A_0 =  \pmat{0 & \bI \\ -\alpha \frac{\partial^4}{\partial \xi^4} & -\beta \frac{\partial^4}{\partial \xi^4} - \gamma}
$ with the domain
\begin{align*}
D(A_0) &= D(\cA)~\cap~\cN (\cB) \\
&= \left \{ (v_1, v_2) \in V_0 \mid \alpha \frac{d^2 v_1}{d\xi^2}  + \beta \frac{d^2 v_2}{d\xi^2 }  \in H^2(0,l), ~~ \frac{d^2 v_1}{d\xi^2}(l) = \frac{d^3 v_1}{d\xi^3}(l) = 0    \right \}.
\end{align*}

\subsection{The extended system}
\label{sec-beam-1stext}

Choose $\Ad = \pmat{0 & \bI \\ -\alpha \frac{\partial^4}{\partial \xi^4} & -\beta \frac{\partial^4}{\partial \xi^4}}$ and \\
$\Arc = \pmat{0 & 0 \\ 0 & -\gamma} $ with $D(\Ad) =  D(\cA)$ and $D(\Arc) = X_0$. 
We construct $Eu(t) = \pmat{g_1(\xi)\\ g_2(\xi)} u(t)$ satisfying both conditions \eqref{ass-cond1} and \eqref{ass-condBC} as follows 
\begin{align*}
\pmat{0 & \bI \\ -\alpha \frac{\partial^4}{\partial \xi^4} & -\beta \frac{\partial^4}{\partial \xi^4}} \pmat{g_1(\xi)\\ g_2(\xi)} &= \eta \pmat{g_1(\xi)\\ g_2(\xi)}, \qquad 
\cB \pmat{g_1(\xi)\\ g_2(\xi)} = 1. \\
g_1(0) = g_1'(0) = g_1'''(l) &= 0, \quad  g_2(0) = g_2'(0) = 0. 
\end{align*} 
We need to solve a system of ODEs as follows
\begin{subequations} \label{eq-beam-1stode}
\begin{align}
g_2(\xi) &= \eta g_1(\xi), \\
g_1''''(\xi) &= \frac{-\eta^2}{\alpha + \beta \eta} g_1(\xi), \\
g_1(0) = g_1'(0) = g_1'''(l) &= 0, \quad g_1''(l) = 1, \\
g_2(0) = g_2'(0)  &= 0.
\end{align}
\end{subequations}
We have freedom of choices on boundary conditions of $g_2$. Here we choose $g_2'''(l) = 0$, and $g_2''(l) = \eta$. The condition $\alpha g''_1 + \beta g_2''   \in H^2(0,l)$ can be verified after solving the system. 

Define the change of variable
$
\pmat{v \\ \dot{v}} = \pmat{w \\ \dot{w}} - Eu(t), 
$
and the new control $\kappa(t) = \dot{u}(t) - \eta u(t)$. The extended system can be rewritten in terms of the new state $x = (v, \dot{v}, u)^\top $in the abstract form $\dot{x}(t) = A x(t) + B\kappa(t) $ where
\begin{align} \label{beam-eq-1stext}
A = \pmat{0 & \bI & 0  \\ 
-\alpha \frac{\partial^4}{\partial \xi^4} & -\beta \frac{\partial^4}{\partial \xi^4} - \gamma & -\gamma g_2(\xi) \\
0 & 0 & \eta \\
}, \quad 
B= \pmat{-g_1(\xi)  \\ -g_2(\xi) \\ 1 }. 
\end{align} 
The sesquilinear associated to the operator $A_0$ is bounded and coercive (see\cite{ItoMor98} and \cite[Section V.C.]{PauPhan19}). Thus the sesquilinear generated by operator $A$ in \eqref{beam-eq-1stext} is also bounded and coercive as shown in the proof of Theorem \ref{the-RORP}.  

The observation part can be rewritten in the new state 
\begin{align*}
y(t) =  C_1 v(\cdot, t) + C_2 \dot{v}(\cdot, t) = C_1 (w (\cdot, t) + g_1(\cdot) u(t) ) + C_2( \dot{w}(\cdot, t) + g_2(\cdot) u (t) )
\end{align*}
which implies that $C = \pmat{C_1 & C_2 & C_1 g_1 + C_2 g_2} $. Under this setting, the extended system $(A,B,C)$ can be rewritten in the abstract form as in \eqref{eq-sys-ext}-\eqref{eq-obs-ext}.

\subsection{An alternative extended system}
\label{sec-beam-2ndext}
For second-order (in time) PDE models, we can use an alternative approach to construct the extended system. For this class of system, this approach here is more natural than the first one. However it still has some disadvantages that we will discuss below. 

Let us define $v(\xi,t) = w(\xi,t) - g(\xi) u(t)$ where $g(\xi)$ solves the ODE  
\begin{subequations} \label{eq-beam-2ndode}
\begin{align}
g'''' (\xi) - \eta g (\xi) &= 0, \\
g(0) = g'(0) = g'''(l) &= 0, \\
g''(l) &= 1
\end{align}
\end{subequations}
where $\eta$ is a positive constant. Then we can rewrite the equation \eqref{eq-beam} as follows 
\begin{subequations} \label{eq-pde-ext}
\begin{align}
\frac{\partial^2 v}{\partial t^2}(\xi,t) &+ \alpha \frac{\partial^4 v}{\partial \xi^4}(\xi,t) + \left( \beta  \frac{\partial^4 }{\partial \xi^4}  + \gamma \right) \frac{\partial v}{\partial t}(\xi,t)  \\
&= - \left( u''(t) + (\beta \eta + \gamma)  u'(t) + \alpha \eta u(t) \right)  g(\xi), \\
v(0,t) &= \frac{\partial v}{\partial \xi} (0,t)  = \frac{\partial^2 v}{\partial \xi^2} (l,t) = \frac{\partial^3 v}{\partial \xi^3} (l,t) = 0. 
\end{align}
\end{subequations} 

Defining $\kappa(t) =  u''(t) + (\beta \eta + \gamma)  u'(t) + \alpha \eta u(t)$, we get an alternative extended system $\dot{x}(t) = A x(t) + B \kappa(t)$ where 
\begin{align} \label{beam-eq-2ndext}
A = \pmat{0 & \bI & 0 & 0 \\ 
-\alpha \frac{\partial^4}{\partial \xi^4} & -\beta \frac{\partial^4}{\partial \xi^4} - \gamma & 0 & 0  \\
0 & 0 & 0 & 1 \\
0 & 0 & -\alpha \eta & - \beta \eta  - \gamma
\\}, \quad  
B = \pmat{ 0 \\ - g(\xi) \\  0 \\  1 }. 
\end{align}
The observation part can be rewritten in the new state as 
\begin{align*}
y (t) = C_1 v(\cdot, t) + C_2 \dot{v}(\cdot, t) = C_1 (w (\cdot, t) + g(\cdot) u(t) ) + C_2( \dot{w}(\cdot, t) + g(\cdot) u' (t) ), 
\end{align*}
which leads to the output operator
$
C = \pmat{C_1 & C_2 & C_1 g & C_2 g}
$.

\begin{lem}
Consider the abstract differential equation $\dot{x}(t) = A x(t) + B \kappa(t)$ where $A$ and $B$ are defined in \eqref{beam-eq-2ndext}. Assume that $u \in C^3([0, \tau]; \R)$ for all $\tau >0$ and $(v_0, v_1) = (w_0 - gu(0), w_1 - gu'(0)) \in D(A_0)$. The extended system with $(x_0)_1 = v_0,~(x_0)_2 = v_1,~(x_0)_3 = u(0),~(x_0)_4 = u'(0)$ has a unique solution $x(t) = (v(t), \dot{v}(t), u(t), u'(t))^\top$. 
\end{lem} 
\begin{proof}
 By denoting 
$A_1 = \pmat{0 & \bI \\ -\alpha \frac{\partial^4}{\partial \xi^4} & -\beta \frac{\partial^4}{\partial \xi^4} - \gamma}$ and $A_2 = \pmat{0 & 1\\ -\alpha \eta & -\beta \eta - \gamma}$, we rewrite $A = \pmat{A_1 & 0 \\ 0 & A_2}$. Then we analogously apply Lemma 3.2.2 in \cite{CurZwa95} and the procedure of Theorem \ref{the-change-var} to get the result. 
\end{proof}

\begin{rem}
Comparing with the case of parabolic equations in Section \ref{sec-para}, the difference is that we can find the extension $Eu(t)$  explicitly by solving ODE \eqref{eq-beam-1stode} or \eqref{eq-beam-2ndode}.
  
Considering the system of ODE \eqref{eq-beam-1stode}, the characteristic equation is $\lambda + \frac{\eta^2}{\alpha + \beta \eta} = 0$.  By denoting $\tilde{\eta} = \sqrt[4]{\frac{\eta^2}{4(\alpha + \beta \eta)}}$, the solution of characteristic equation is $\lambda = \pm \tilde{\eta} \pm i \tilde{\eta}$. Thus the general solution is 
\begin{align*}
g_1 (\xi) = m_1 \mathrm{e}^{\tilde{\eta} \xi} \cos (\tilde{\eta} \xi) + 
m_2 \mathrm{e}^{\tilde{\eta} \xi} \sin (\tilde{\eta} \xi) + 
m_3 \mathrm{e}^{-\tilde{\eta} \xi} \cos (\tilde{\eta} \xi) + 
m_4 \mathrm{e}^{-\tilde{\eta} \xi} \sin (\tilde{\eta} \xi)
\end{align*}
and $g_2(\xi) = \eta g_1(\xi)$. Obviously $g_1 (\xi)$ and $g_2(\xi)$ belong to $H^2(0,l)$.
On the other hand, for the ODE \eqref{eq-beam-2ndode}, the corresponding characteristic equation is 
$
\lambda^4 - \eta  = 0 
$ 
whose solutions are $\lambda =\pm \sqrt[4]{\eta}$ and $\lambda =\pm i \sqrt[4]{\eta}$. The general solution is 
\begin{align*}
g(\xi) = m_1 \mathrm{e}^{\sqrt[4]{\eta} \xi} + m_2 \mathrm{e}^{-\sqrt[4]{\eta} \xi} + m_3 \cos(\sqrt[4]{\eta} \xi) + m_4 \sin(\sqrt[4]{\eta} \xi).
\end{align*}

All unknown parameters $m_1,~m_2,~m_3,~m_4$ can be determined from the boundary conditions by solving a corresponding linear algebraic system.

\subsection{Two approaches with other types of boundary control}
The type of boundary condition below was presented before in some works \cite[Section 3]{ItoMor98} or \cite[Section V.C]{PauPhan19}. Here we design a boundary control. 
The construction of extension operator $Eu(t)$ in section \ref{sec-beam-1stext} can be modified to adapt with this type of boundary condition
\begin{subequations}
\begin{align} \label{eq-beam-anotherBC}
w(0,t) = \frac{\partial w}{\partial \xi} (0,t) &= 0, \\  
 \alpha \frac{\partial^2 w}{\partial \xi^2} (l,t) + \beta \frac{\partial^3 w}{\partial \xi^2 \partial t} (l,t)  &= u(t),\\
 \alpha \frac{\partial^3 w}{\partial \xi^3} (l,t) + \beta \frac{\partial^4 w}{\partial \xi^3 \partial t} (l,t)    &= 0,
\end{align}
\end{subequations}
By denoting $M(\xi, t) = \alpha \frac{\partial^2 w}{\partial \xi^2} (\xi,t) + \beta \frac{\partial^3 w}{\partial \xi^2 \partial t} (\xi,t) $, we modify the domain of $\cA$ as
\begin{align*}
D(\cA) &= \left \{ (w_1, w_2) \in V \mid \alpha \frac{d^2 w_1}{d\xi^2}  + \beta \frac{d^2 w_2}{d\xi^2}  \in H^2(0,l), ~~ \frac{d M}{d \xi} (l,\cdot) = 0    \right \}.
\end{align*}
The boundary operator $\cB:  X_0 \to \C$ denotes by $\cB \pmat{w_1 \\ w_2} = \alpha \frac{d^2}{d\xi^2} w_1(l) + \beta \frac{d^2}{d\xi^2} w_2(l)$ with $D(\cB) = D(\cA)$. \\
The domain of operator $A_0$ is denoted by 
\begin{align*}
D(A_0) = \left \{ (v_1, v_2) \in V \mid \alpha \frac{d^2 v_1 }{d\xi^2} + \beta \frac{d^2 v_2}{d\xi^2}  \in H^2(0,l), ~~ M (l, \cdot) = \frac{d  M}{d \xi}(l,\cdot) = 0    \right \}.
\end{align*}
With the same choice of $\Ad$ and $\Arc$, we get the system of ODEs as follows 
\begin{align*}
g_2(\xi) = \eta g_1(\xi), \quad  g_1''''(\xi) = \frac{-\eta^2}{\alpha + \beta \eta} g_1(\xi)
\end{align*}
whose boundary conditions are modified as 
\begin{align*}
g_1(0) = g_1'(0) = g_1'''(l) &= 0,\quad g_1''(l) = \frac{1}{\alpha + \beta \eta}, \\
g_2(0) = g_2'(0) &= 0.
\end{align*}
Again, we can choose $g_2'''(l) = 0$ and $g_2''(l) = \frac{\eta}{\alpha + \beta \eta}$.

However the approach in section \ref{sec-beam-2ndext} does not work with this type of boundary conditions.

\end{rem}

\subsection{A numerical example}
 
In this example, we consider the system \eqref{eq-beam} with $l=7, \alpha = 10, \beta = 0.01$, and $\gamma = 10^{-5}$. The observation is 
\begin{align*}
y(t) = \int_2^4 {w(\xi,t) + \frac{\partial w}{\partial t} (\xi,t) d\xi}, \quad \text{i.e.} \quad C_1 = C_2 = \chi_{(2,4)}(\cdot). 
\end{align*}
With the choice of parameters, the stability margin of the system is very small (approximately $10^{-3}$). In this example, we use the boundary control to improve the stability of the original system and obtain an acceptable closed-loop stability margin. 

We want to track the reference signal $\yref(t) = \frac{1}{10}(t^2 - t)\sin(3t)$. The set of frequency has only one element $\{ 3 \}$ with $n_k = 3$. 

We also used two different meshes. Again, we use Finite Element Method with cubic Hermit shape functions as in \cite[Section V.C.]{PauPhan19}. 
We construct the observer-based finite-dimensional controller based on the algorithm in Section \ref{sec-DesignCon} using a coarse mesh with $N = 34$ (the corresponding size of the matrix $A^N$ is 138) and subsequent Balanced Truncation with order $r = 50$. The internal model has dimension $\dim Z_0 = 2 \times 3 = 6$. 

For the controller in Section \ref{sec-beam-1stext}, we choose $\eta = 0.12$ in system \eqref{eq-beam-1stode}. The corresponding solutions $g_1$ and $g_2$ are plotted in Figure \ref{fig_extbeam1}. The parameters of the stabilization are chosen as 
\begin{align*}
\alpha_1 = 0.65, \quad \alpha_2 = 0.5, \quad R_1 =  0.1, \quad R_2 = 1. 
\end{align*}
For the alternative extended system in Section \ref{sec-beam-2ndext}, we choose $\eta = 10$ in \eqref{eq-beam-2ndode}. The solution $g$ of \eqref{eq-beam-2ndode} with $\eta = 10$ is plotted in Figure \ref{fig_extbeam2}.
We choose other parameters of stabilization to improve the stability margin as
\begin{align*}
\alpha_1 = 0.75, \quad \alpha_2 = 0.5, \quad R_1 =  1, \quad R_2 = 10^{-3}. 
\end{align*}

\begin{figure}[tbhp]
\centering
\subfloat[The functions $g_1$ and $g_2$ with $\eta = 0.12$.]{\label{fig_extbeam1}\includegraphics[width=.49\textwidth]{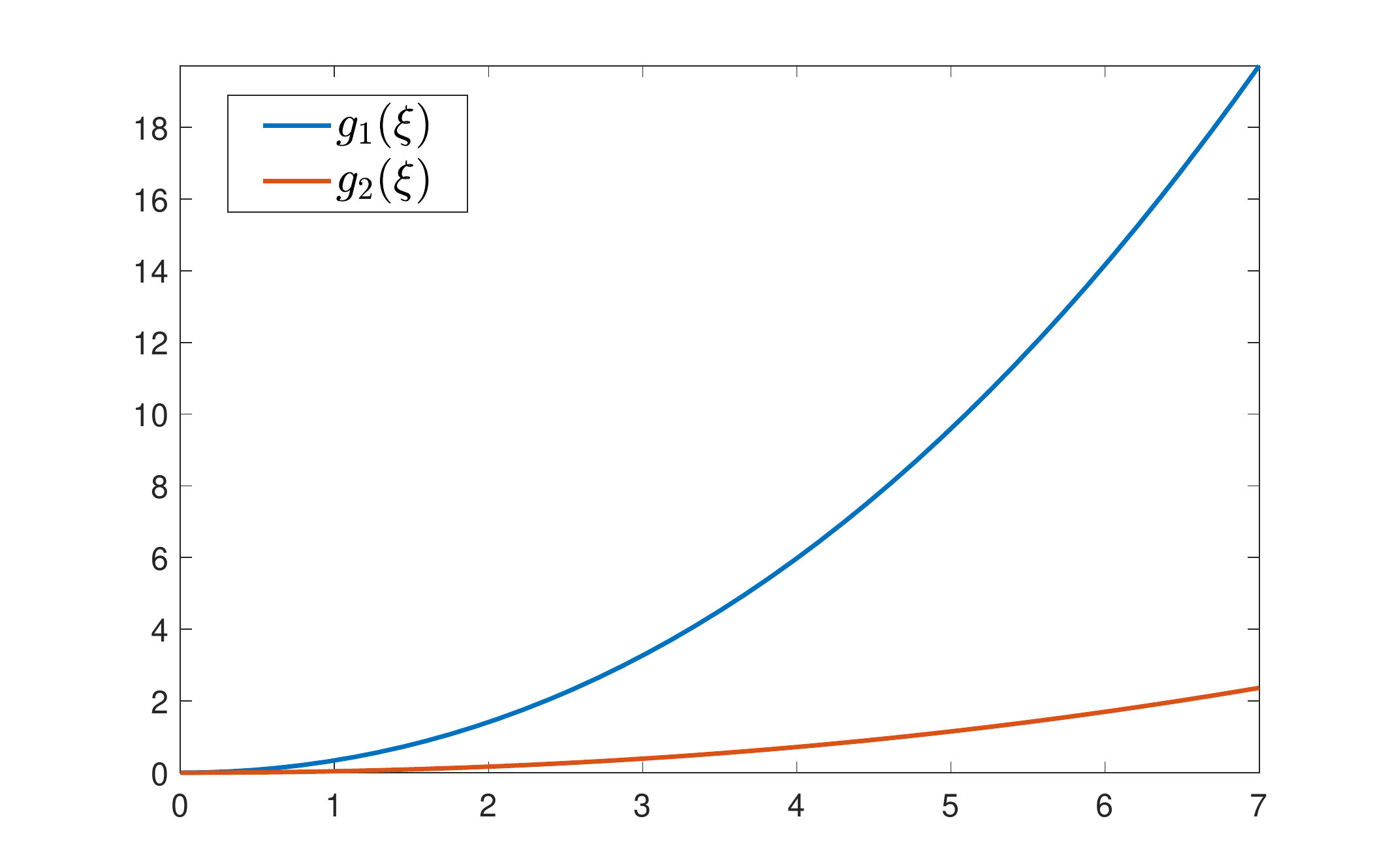}}
\subfloat[The function $g$ with $\eta = 10$.]{\label{fig_extbeam2}\includegraphics[width=.49\textwidth]{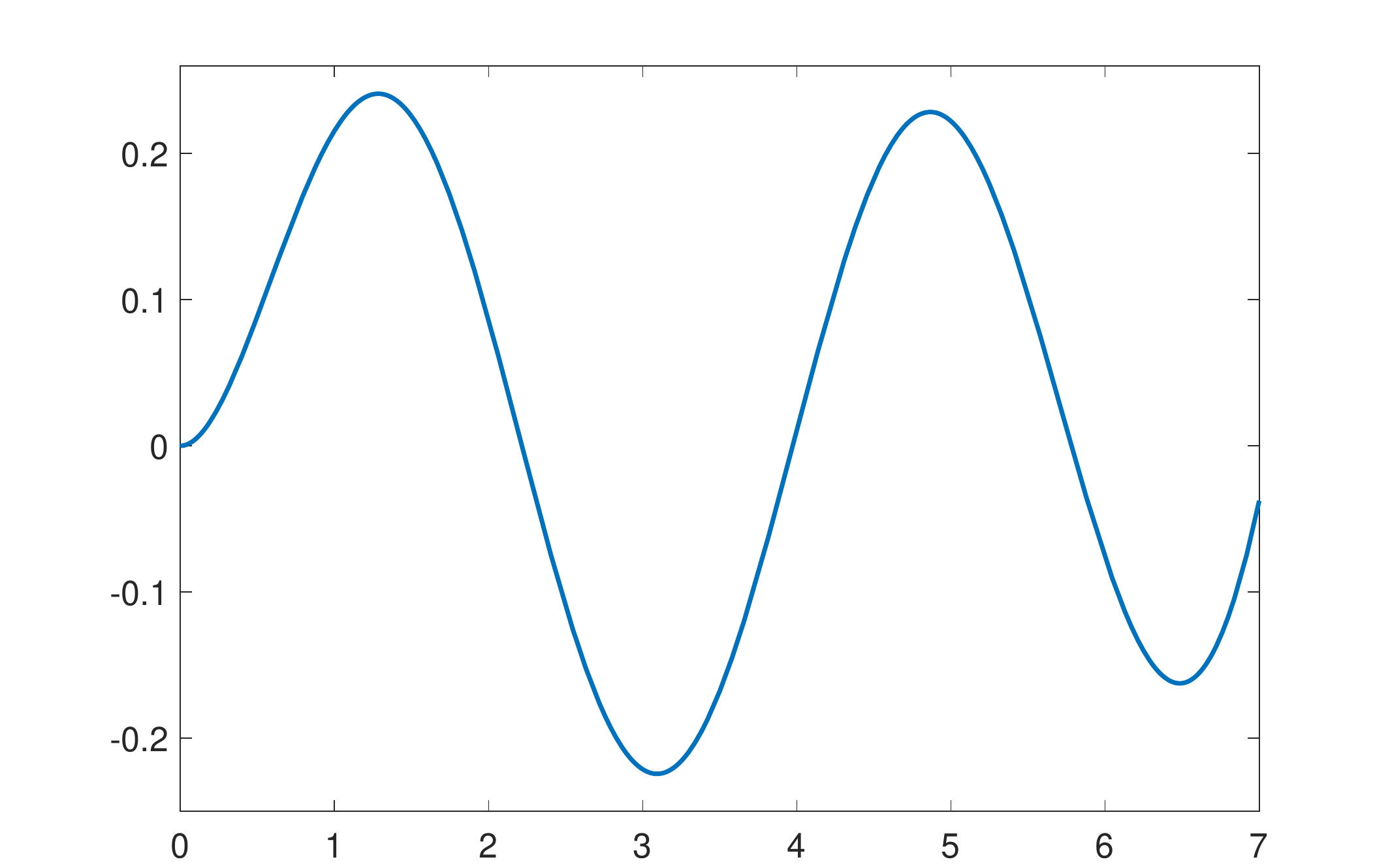}}
\caption{Solutions of ODEs.}
\label{fig:testfig}
\end{figure}

For the simulation of the original system \eqref{eq-beam}, we use another Finite Element approximation with $M = 86$. The corresponding size of matrix $A^M$ is 346. The initial state of the original systems $v_0 (\xi) = 0.25 (\cos (5 \xi) - 2)$, $v_1 (\xi) = 0.25 \sin (5\xi) $, and $z_0 =\{ v \in \R^{6+40} \mid v_i = -0.3 \}$. 

The tracking controlled signals under two different extensions are plotted in Figure \ref{fig_trackbeam} where the blue line corresponds with the extension \eqref{beam-eq-1stext} and the green one corresponds with the extension \eqref{beam-eq-2ndext}. 

\begin{figure}[]  
\centering
\includegraphics[width=.8\textwidth]{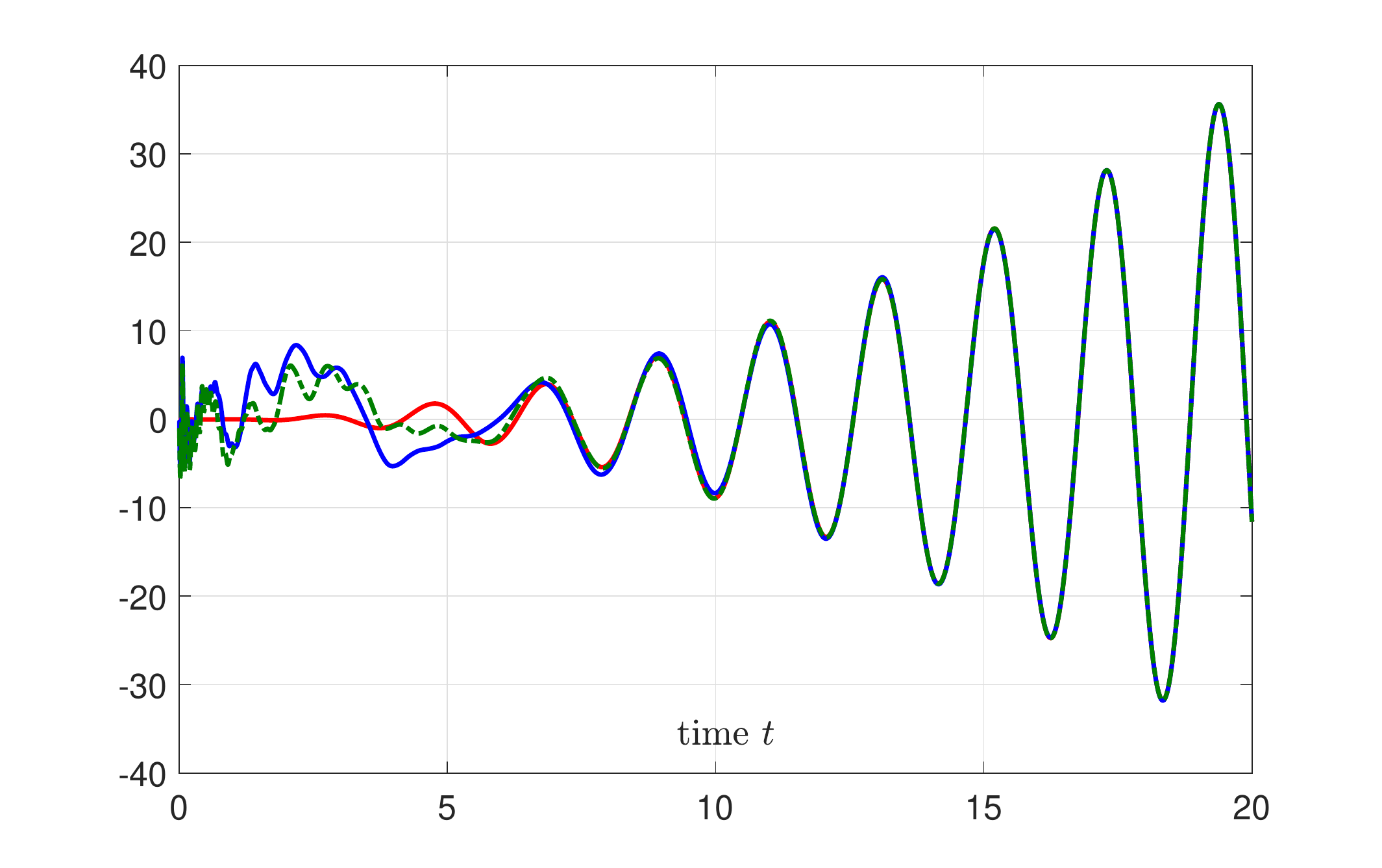}
\caption{Output tracking of the boundary controlled beam equation with two different extensions.}
\label{fig_trackbeam}
\end{figure}

\section{Final remarks}
We have presented  new methods for design finite-dimensional reduced order controllers for robust output regulation problems of boundary control systems. The controllers are constructed based on an extended system. Theorem \ref{the-RORP} shows that the controllers solve the robust output regulation problem. The construction of extended system is completed by two additional assumptions. Comparing with the choice of arbitrary right inverse operators in the literature, our construction is efficient in PDE models with multi-dimensional domains. Concerning with the boundary disturbance signals, some examples was also introduced before in \cite[Section V. A.]{PauPhan19} or \cite{Guo14}. We remark that the method can be analogously applied to construct a new bounded disturbance operator. We can then extend the control design here for the case with boundary disturbance signals. 

We must assume the boundedness of output operators because the extension approach does not have an analogue for the output operators. Moreover, the controller design method in \cite{PauPhan19} requires a bounded output operator, and extending the results for unbounded $C$ is an important topic for future research. 

As shown in the proofs in~\cite{PauPhan19}, the possibility for model reduction in the controller design (for a fixed $N\in \N$) is based on the smallness of the $H_\infty$-error between the transfer functions of the stable finite-dimensional systems $(A_L^r,[B_L^r,L^r],K_2^r)$ and $(A^N+L^NC^N,[B^N+L^ND^N,L^N],K_2^N)$. Our results do not provide lower bounds for a suitable value of $r$, but the results on Balanced Truncation show that for a given $r\leq N$ the error between these transfer functions is determined by the rate of decay of the Hankel singular values of the latter system. Because of this, rapid decay of the Hankel singular values of $(A^N+L^NC^N,[B^N+L^ND^N,L^N],K_2^N)$ can be used as an indicator that reduction of the controller order is possible for the considered system and its approximation $(A^N,B^N,C^N)$.

\section*{Acknowledgments} 
The research is supported by the Academy of Finland grants number 298182 and 310489 held by L. Paunonen. D. Phan is partially supported by Universit\"at Innsbruck. 

\providecommand{\href}[2]{#2}
\providecommand{\arxiv}[1]{\href{http://arxiv.org/abs/#1}{arXiv:#1}}
\providecommand{\url}[1]{\texttt{#1}}
\providecommand{\urlprefix}{URL }


\begin{thebibliography}{10}

\bibitem{Bad09}
\newblock M.~Badra,
\newblock {Feedback stabilization of the 2-D and 3-D Navier-Stokes equations
  based on an extended system},
\newblock \emph{ESAIM: COCV}, \textbf{15} (2009), 934--968,
\newblock \urlprefix\url{https://doi.org/10.1051/cocv:2008059}.

\bibitem{BanKun84}
\newblock H.~T. Banks and K.~Kunisch,
\newblock {The Linear Regulator Problem for Parabolic Systems},
\newblock \emph{SIAM Journal on Control and Optimization}, \textbf{22} (1984),
  684--698.

\bibitem{CurZwa95}
\newblock R.~F. Curtain and H.~Zwart,
\newblock \emph{{An Introduction to Infinite--Dimensional Linear Systems
  Theory}}, vol.~21 of Texts in Applied Mathematics,
\newblock Springer-Verlag New York, 1995.

\bibitem{Guo14}
\newblock B.-Z. Guo, H.-C. Zhou, A.~S. AL-Fhaid, A.~M.~M. Younas and A.~Asiri,
\newblock {Stabilization of Euler-Bernoulli Beam Equation with Boundary Moment
  Control and Disturbance by Active Disturbance Rejection Control and Sliding
  Mode Control Approaches},
\newblock \emph{Journal of Dynamical and Control Systems}, \textbf{20} (2014),
  539--558.

\bibitem{HamPohMMAR02}
\newblock T.~H{\"a}m{\"a}l{\"a}inen and S.~Pohjolainen,
\newblock Robust regulation for exponentially stable boundary control systems
  in {H}ilbert space,
\newblock in \emph{Proceedings of the 8th IEEE International Conference on
  Methods and Models in Automation and Robotics},
\newblock Szczecin, Poland, 2002,
\newblock 171--178.

\bibitem{HamPoh10}
\newblock T.~H{\"a}m{\"a}l{\"a}inen and S.~Pohjolainen,
\newblock Robust regulation of distributed parameter systems with
  infinite-dimensional exosystems,
\newblock \emph{SIAM J. Control Optim.}, \textbf{48} (2010), 4846--4873.

\bibitem{Imm07a}
\newblock E.~Immonen,
\newblock On the internal model structure for infinite-dimensional systems:
  {T}wo common controller types and repetitive control,
\newblock \emph{SIAM J. Control Optim.}, \textbf{45} (2007), 2065--2093.

\bibitem{ItoMor98}
\newblock K.~Ito and K.~Morris,
\newblock {An Approximation Theory of Solutions to Operator Riccati Equations
  for $H^\infty$ Control},
\newblock \emph{SIAM J. Control Optim.}, \textbf{36} (1998), 82--99.

\bibitem{LogTow97}
\newblock H.~Logemann and S.~Townley,
\newblock Low-gain control of uncertain regular linear systems,
\newblock \emph{SIAM J. Control Optim.}, \textbf{35} (1997), 78--116.

\bibitem{PauPhan19}
\newblock L.~{Paunonen} and D.~{Phan},
\newblock Reduced order controller design for robust output regulation,
\newblock \emph{IEEE Transactions on Automatic Control}, 1--1.

\bibitem{Pau16a}
\newblock L.~Paunonen,
\newblock Controller design for robust output regulation of regular linear
  systems,
\newblock \emph{IEEE Trans. Automat. Control}, \textbf{61} (2016), 2974--2986.

\bibitem{PhanRod18}
\newblock D.~Phan and S.~S. Rodrigues,
\newblock {Stabilization to trajectories for parabolic equations},
\newblock \emph{Mathematics of Control, Signals, and Systems}, \textbf{30}
  (2018), 11.

\bibitem{RebWei03}
\newblock R.~Rebarber and G.~Weiss,
\newblock Internal model based tracking and disturbance rejection for stable
  well-posed systems,
\newblock \emph{Automatica J. IFAC}, \textbf{39} (2003), 1555--1569.

\bibitem{Rod15}
\newblock S.~S. Rodrigues,
\newblock {Boundary observability inequalities for the 3D Oseen-Stokes system
  and applications},
\newblock \emph{ESAIM: COCV}, \textbf{21} (2015), 723--756,
\newblock \urlprefix\url{https://doi.org/10.1051/cocv/2014045}.

\bibitem{Sal87}
\newblock D.~Salamon,
\newblock Infinite-dimensional linear systems with unbounded control and
  observation: {A} functional analytic approach,
\newblock \emph{Trans. Amer. Math. Soc.}, \textbf{300} (1987), 383--431.

\bibitem{Sta05}
\newblock O.~Staffans,
\newblock \emph{Well-Posed Linear Systems},
\newblock Encyclopedia of Mathematics and its Applications, Cambridge
  University Press, 2005.

\bibitem{TucWei09}
\newblock M.~Tucsnak and G.~Weiss,
\newblock \emph{{Observation and Control for Operator Semigroups}},
\newblock {Birkh\"auser Basel}, 2009.

\end{thebibliography}
\end{document}